\documentclass{article}
\usepackage[T1]{fontenc}
\usepackage{amssymb,amsmath,amsthm,enumerate,xr}
\usepackage{algorithm}
\usepackage{algorithmic}
\usepackage{stmaryrd}
\usepackage{dsfont}
\usepackage{version}
\usepackage{subfig}
\usepackage{ifthen}
\usepackage{graphicx,color}
\usepackage{algorithm,algorithmic}
\newcommand{\1}{\ensuremath{\mathbf{1}}}

\theoremstyle{plain} \newtheorem{theorem}{Theorem}[section]
\theoremstyle{plain} \newtheorem{proposition}[theorem]{Proposition}
\theoremstyle{plain} 
\theoremstyle{plain} \newtheorem{lemma}[theorem]{Lemma}
\theoremstyle{remark} \newtheorem{remark}[theorem]{Remark}

\newcounter{hypA}
\newenvironment{hypA}{\refstepcounter{hypA}\begin{itemize}
\item[{\bf H\arabic{hypA}}]}{\end{itemize}}

\newcounter{hypS}
\newenvironment{hypS}{\refstepcounter{hypS}\begin{itemize}
\item[{\bf S\arabic{hypS}}]}{\end{itemize}}

\newcommand{\osc}{\mathrm{osc}}
\newcommand{\rme}{\mathrm{e}}
\newcommand{\eqdef}{\ensuremath{\stackrel{\mathrm{def}}{=}}}

\def\Xset{\mathbb{X}}
\def\Yset{\mathbb{Y}}
\def\Zset{\mathbb{Z}}
\def\sigmaX{\mathcal{X}}
\def\sigmaY{\mathcal{Y}}

\newcommand{\pscal}[2]{\left\langle #1, #2 \right\rangle}
\newcommand{\K}{\mathcal{K}}

\newcommand{\limE}[2]{\mathsf{S}(#1,#2)}
\newcommand{\Rset}{\mathbb{R}}
\newcommand{\Nset}{\mathbb{N}}

\newcommand{\eqsp}{\;}
\newcommand{\jacG}{\Gamma}
\newcommand{\param}{\theta}
\newcommand{\paramset}{\Theta}
\newcommand{\rmd}{\mathrm{d}}
\newcommand{\rmB}{\mathrm{B}}
\newcommand{\rmF}{\mathrm{F}}
\newcommand{\rmL}{\mathrm{L}}
\newcommand{\bfy}{\mathbf{y}}
\newcommand{\bfY}{\mathbf{Y}}

\newcommand{\Expparam}[2]{\mathbb{E}_{#2}\left[#1\right]}

\newcommand{\filter}[4]{\phi_{#3\vert #4}^{#1,#2}}

\newcommand{\PP}{\mathbb{P}}
\newcommand{\Staset}{\mathcal{L}}

\newcommand{\loglik}[2]{\ell_{#2}^{#1}}

\newcommand{\CE}[4][]
{
\ifthenelse{\equal{#1}{}}{\mathbb{E}^{#2}_{#3}\left[#4\right]}{\mathbb{E}^{#2}_{#3}\left[#4\middle | #1\right]}
}

\newcommand{\lpnorm}[2]{\ensuremath{\left\| #1 \right\|_{#2}}} 

\newcommand{\smoothfunc}[2]{\Phi_{#2}^{#1}}
\newcommand{\lyap}{\mathrm{W}}

\newcommand{\ps}[1]{#1\mathrm{-a.s.}}
\newcommand{\limEMmap}{\mathrm{G}}
\newcommand{\limEMmapt}{\mathrm{R}}
\newcommand{\mapS}{\bar{\mathrm{S}}}
\newcommand{\PPim}{\PP}
\newcommand{\Sset}{\mathcal{S}}
\newcommand{\sigmafield}[1]{\mathcal{B}(#1)}

\newcommand{\normTV}[1]{\left |\left | #1 \right| \right |_{\mathrm{TV}}}

\newcommand{\shift}{\vartheta}

\newcommand{\averageparam}[1][]%
{\ifthenelse{\equal{#1}{}}{\ensuremath{\widetilde{\theta}}}{\ensuremath{\widetilde{\theta}}_{#1}}
}

\newcommand{\mix}[1][]%
{\ifthenelse{\equal{#1}{a}}{\alpha}{\beta}
}

\begin{document}
  
  \title{Supplement paper to ``Online Expectation Maximization based algorithms
    for inference in Hidden Markov Models''}
 \author{Sylvain Le Corff\,\footnote{LTCI, CNRS and TELECOM ParisTech, 46 rue Barrault 75634 Paris Cedex 13, France. sylvain.lecorff@telecom-paristech.fr}\;\,\footnote{This
    work is partially supported by the French National Research Agency, under
    the programs ANR-08-BLAN-0218 BigMC and ANR-07-ROBO-0002.}\, and Gersende Fort\,\footnote{LTCI, CNRS and TELECOM ParisTech, 46 rue Barrault 75634 Paris Cedex 13, France. gersende.fort@telecom-paristech.fr}}

\maketitle

This is a supplementary material to the paper~\cite{lecorff:fort:2011}.

It contains technical discussions and/or results adapted from published papers. In Sections~\ref{BOEMsupp:sec:proof:aux} and
\ref{BOEMsupp:sec:forgetting:HMM}, we provide results - useful for the proofs of some 
theorems in~\cite{lecorff:fort:2011} - which are close to existing results in
the literature.
  
It also contains, in Section~\ref{BOEMsupp:sec:appli}, additional plots for the
numerical analyses in \cite[Section~$3$]{lecorff:fort:2011}.

To make this supplement paper as self-contained as possible, we decided to
rewrite in Section~\ref{BOEMsupp:sec:assum} the model and the main definitions
introduced in~\cite{lecorff:fort:2011}.

\section{Assumptions and Model}
\label{BOEMsupp:sec:assum}
Our model is defined as follows. Let $\paramset$ be a compact subset of $\Rset^{d_\theta}$. We are given a family of transition kernels
$\{M_{\param}\}_{\param\in\paramset}$, $M_{\param}:\Xset\times\sigmaX\to[0,1]$, a positive $\sigma$-finite
measure $\mu$ on $(\Yset,\sigmaY)$, and a family of transition densities with respect to $\mu$, $\{g_{\param}\}_{\param\in\paramset}$,  $g_{\param}:\Xset\times\Yset\to\mathbb{R}_+$. It is assumed that, for any $\param\in\paramset$ and any $x\in\Xset$, $M_{\param}(x,\cdot)$ has a density $m_{\param}(x,\cdot)$ with respect to a finite measure $\lambda$ on $(\Xset,\sigmaX)$.
In the setting of this paper, we consider  a single observation path $\bfY\eqdef\{Y_t\}_{t\in\Zset}$ defined on the probability space $\left(\Omega,\mathcal{F},\PPim\right)$ and taking values in $\Yset^{\Zset}$. The following assumptions are assumed to hold.

\begin{hypA} 
\label{BOEMsupp:assum:exp}
  \begin{enumerate}[(a)]
  \item \label{BOEMsupp:assum:exp:decomp}There exist continuous functions $\phi : \paramset \to \Rset$, $\psi : \paramset
    \to \Rset^d$ and $S: \Xset \times \Xset \times \Yset \to \Rset^d$
    s.t.
\begin{equation*}
  \label{BOEMsupp:eq:exponential:family}
  \log m_\param(x,x') + \log g_\param(x',y) = \phi(\param) +
  \pscal{S(x,x,',y)}{\psi(\param)} \eqsp,
\end{equation*}
where $\pscal{\cdot}{\cdot}$ denotes the scalar product on $\Rset^d$. 
\item \label{BOEMsupp:assum:exp:convex}There exists an open subset $\Sset$ of $\Rset^d$
  that contains the convex hull of $S(\Xset \times \Xset \times \Yset)$.
    \item \label{BOEMsupp:assum:exp:max}There exists a continuous function $\bar \param :
  \Sset \to \paramset$ s.t. for any $s \in \Sset$,
\[
\bar \param(s) = \mathrm{argmax}_{\param \in \paramset} \; \left\{ \phi(\param) +
  \pscal{s}{\psi(\param)} \right\} \eqsp.
\]
  \end{enumerate}
\end{hypA}
\begin{hypA}\label{BOEMsupp:assum:strong}
  There exist $\sigma_{-}$ and $\sigma_{+}$ s.t. for any
  $\left(x,x^{\prime}\right)\in\Xset^2$ and any $\param \in \paramset$, $0
  <\sigma_{-} \leq m_\param(x,x^{\prime})\leq \sigma_{+}$.  Set $\rho \eqdef 1 -
  (\sigma_-/\sigma_+)\eqsp.$
\end{hypA}

We now introduce assumptions on the observation process. For any sequence of r.v.  $Z\eqdef\{Z_{t}\}_{t\in\Zset}$ on
$(\Omega,\PPim,\mathcal{F})$, let
\begin{equation*}
\mathcal{F}_{k}^{Z} \eqdef \sigma\left(\{Z_{u}\}_{u\leq k}\right)\quad \mbox{and}\quad \mathcal{G}_{k}^{Z} \eqdef \sigma\left(\{Z_{u}\}_{u\geq k}\right)
\end{equation*}
be $\sigma$-fields associated to $Z$. We also define the mixing coefficients
by, see \cite{davidson:1994},
\begin{equation}
\label{BOEMsupp:eq:def:mixing}
\mix^{\bfY}(n)=\underset{u\in\Zset}{\sup}\,\underset{B\in\mathcal{G}_{u+n}^{\bfY}}{\sup}\,\CE[]{}{}{|\PP(B| \mathcal{F}_{u}^{\bfY}) - \PP(B)|}\eqsp, \forall\; n\geq0\eqsp.
\end{equation} 

\begin{hypA}\label{BOEMsupp:assum:moment:sup} 
\hspace{-0.15cm}-($p$)  $\CE[]{}{}{\sup_{x,x' \in \Xset^2} \, |S(x,x',\bfY_{0})|^{p}}<+\infty$.
\end{hypA}
\begin{hypA}\label{BOEMsupp:assum:obs}
\begin{enumerate}[(a)]
\item \label{BOEMsupp:assum:obs:mix}$\bfY$ is a $\beta$-mixing stationary sequence such that there exist $C\in[0,1)$ and $\mix\in(0,1)$ satisfying, for any $n\geq 0$,
  $\mix^{\bfY}(n) \leq C\mix^{n}$, where $\mix^{\bfY}$ is defined in
  \eqref{BOEMsupp:eq:def:mixing}.
  \item \label{BOEMsupp:assum:obs:b+:b-} $\CE[]{}{}{|\log b_-(\bfY_0)| +|\log b_+(\bfY_0)| }<+\infty$ where 
    \begin{eqnarray}
      b_-(y) & \eqdef  \inf_{\param \in \paramset} \int g_\param(x,y) \lambda(\rmd x)  \eqsp,\label{BOEMsupp:eq:b-}\\
  b_+(y) & \eqdef  \sup_{\param \in \paramset} \int g_\param(x,y) \lambda(\rmd x) \eqsp.\label{BOEMsupp:eq:b+}
    \end{eqnarray}
\end{enumerate}
\end{hypA}

\begin{hypA}\label{BOEMsupp:assum:size-block}
There exists $c>0$ and $a>1$ such that for all $n\ge 1$, $\tau_{n} = \lfloor cn^{a}\rfloor$.
\end{hypA}

Recall the following definition from \cite{lecorff:fort:2011}: for a
distribution $\chi$ on $(\Xset,\sigmafield{\Xset})$, positive integers
$T,\tau$ and $\param \in \paramset$, set
\begin{equation}
  \label{BOEMsupp:eq:rewrite:barS}
  \bar S_{\tau}^{\chi,T}(\param, \bfY) \eqdef
\frac{1}{\tau} \sum_{t=T+1}^{T+\tau}\smoothfunc{\chi,T}{\param,t,T+\tau}\left(
  S, \bfY\right)\eqsp,
\end{equation}
where $S$ is the function given by
H\ref{BOEMsupp:assum:exp}\eqref{BOEMsupp:assum:exp:decomp} and
\begin{multline}
\label{BOEMsupp:eq:define-Phi:theta}
\smoothfunc{\chi,r}{\param,s,t}(S,\bfy)  \\
\eqdef \frac{\int \chi(\rmd x_r) \{ \prod_{i=r}^{t-1} m_\param(x_i,
  x_{i+1})g_\param(x_{i+1},y_{i+1}) \} \, S(x_{s-1},x_{s},y_{s}) \, 
  \lambda(\rmd x_{r+1:t})}{\int \chi(\rmd x_r) \{\prod_{i=r}^{t-1} m_\param(x_i,
  x_{i+1})g_\param(x_{i+1},y_{i+1}) \}\, \lambda(\rmd x_{r+1:t})}\eqsp.
\end{multline}
We also write $S_{n-1}\eqdef \bar S_{\tau_n}^{\chi_{n-1},T_{n-1}}(\param_{n-1}, \bfY)$ the intermediate quantity computed by the BOEM algorithm in block $n$ and $\widetilde S_{n-1}$ the associated Monte Carlo approximation.

\begin{hypA}\hspace{-0.1cm}-($p$)\label{BOEMsupp:assum:SMCapprox}
There exists $b\ge (a+1)/2a$ (where $a$ is defined in H\ref{BOEMsupp:assum:size-block}) such that, for any $n\ge 0$,
 \[
  \lpnorm{S_{n} - \widetilde S_{n}}{p} = O(\tau_{n+1}^{-b})\eqsp,
  \]
  where $ \widetilde S_{n}$ is the Monte Carlo approximation of $S_{n}$.
\end{hypA}

Define for any $\param \in \paramset$,
\begin{equation}
\label{BOEMsupp:eq:mapS}
\mapS(\param) \eqdef \CE[]{}{}{\limE{\param}{\bfY}}\eqsp.
\end{equation}
\begin{equation}
  \label{BOEMsupp:eq:EM:algorithm}
\limEMmapt(\param) \eqdef  \bar \param \left( \mapS(\param)\right) \eqsp.
\end{equation}
\begin{equation}
 \label{BOEMsupp:eq:mapping:limitingEM}
 \limEMmap(s) \eqdef \mapS(\bar\param(s))
 \eqsp, \qquad \forall s\in\Sset \eqsp,
\end{equation}
where $\bar \param$ is given by
H\ref{BOEMsupp:assum:exp}\eqref{BOEMsupp:assum:exp:max}.

\begin{hypA}\label{BOEMsupp:assum:stable:fixpoint}
  \begin{enumerate}[(a)]
  \item $\mapS$ and $\bar\param$ are twice continuously differentiable on $\paramset$ and $\Sset$.
  \item There exists $0<\gamma<1$ s.t. the spectral radius of $\nabla_{s}(\mapS\circ\bar\param)_{s=\mapS(\param_{\star})}$ is lower than  $\gamma$.
  \end{enumerate}
 \end{hypA} 
 %Note that under H\ref{BOEMsupp:assum:stable:fixpoint}, $\mathrm{sp}(\Gamma)<\gamma$, where $\Gamma \eqdef \nabla \limEMmap (s_{\star})$ and $s_{\star} = \mapS(\param_{\star})$.
Set
\[
T_n \eqdef \sum_{i=1}^n \tau_i \eqsp, \qquad T_0 \eqdef 0 \eqsp.
\]

\section{Detailed proofs of \cite{lecorff:fort:2011}}
\label{BOEMsupp:sec:proof:aux}
\subsection{Proof of \cite[Theorem~$4.4$]{lecorff:fort:2011}}
\label{BOEMsupp:sec:th:bonem:conv}
\begin{proof}
By \cite[Proposition~A.$1$]{lecorff:fort:2011}, it is sufficient to prove that
\begin{equation}
\label{BOEMsupp:eq:conv1}
\left|\lyap\circ\limEMmapt(\param_{n})-\lyap\circ\bar\param (\widetilde S_{n})\right|\underset{n\to +\infty}{\longrightarrow}0\eqsp,\quad\ps{\PPim}\eqsp.
\end{equation}
By Theorem~$4.1$, the function $\mapS$ given by \eqref{BOEMsupp:eq:mapS} is
continuous on $\paramset$ and then $\mapS(\paramset)\eqdef\{s\in\Sset; \exists
\param\in\paramset, s = \mapS(\param)\}$ is compact and, for any $\delta>0$, we can define the compact subset $\mapS(\paramset,\delta) \eqdef
\left\{s\in\Rset^{d}; \rmd(s,\mapS(\paramset))\leq \delta\right\}$ of $\Sset$, where
$\rmd(s,\mapS(\paramset)) \eqdef \inf_{s'\in\mapS(\paramset)}\left|s-s'\right|$. Let
$\delta>0$ (small enough) and $\varepsilon>0$.  Since $\lyap\circ\bar\param$ is
continuous (see H\ref{BOEMsupp:assum:exp}\eqref{BOEMsupp:assum:exp:max} and
\cite[Proposition~$4.2$]{lecorff:fort:2011}) and $\mapS(\paramset,\delta)$ is compact,
$\lyap\circ\bar\param$ is uniformly continuous on $\mapS(\paramset,\delta)$ and there
exists $\eta>0$ s.t.,
\begin{equation}
\label{eq:onestep:conv}
\forall x,y\in\mapS(\paramset,\delta)\eqsp,\quad |x-y|\leq \eta\Rightarrow |\lyap\circ\bar \param(x)-\lyap\circ\bar \param(y)|\leq \varepsilon\eqsp.
\end{equation}
Set $\alpha \eqdef \delta\wedge \eta$ and $\Delta S_n \eqdef
|\mapS(\param_{n})-\widetilde S_{n}|$. We write,
\begin{align*}
  \PPim\left\{\left|\lyap\circ\bar\param (\mapS(\param_{n}))-\lyap\circ\bar\param (\widetilde S_{n})\right|\geq \varepsilon\right\}&\\
  &\hspace{-4cm}=\PPim\left\{\left|\lyap\circ\bar\param (\mapS(\param_{n}))-\lyap\circ\bar\param (\widetilde S_{n})\right|\geq \varepsilon; \Delta S_n> \delta\right\}\\
  &\hspace{-1.5cm}+\PPim\left\{\left|\lyap\circ\bar\param (\mapS(\param_{n}))-\lyap\circ\bar\param (\widetilde S_{n})\right|\geq \varepsilon; \Delta S_n \leq \delta\right\}\\
  &\hspace{-4cm}\leq \PPim\left\{\Delta S_n > \delta\right\} + \PPim\left\{ \Delta S_n >
    \eta\right\} \leq 2\PPim\left\{ \Delta S_n > \alpha\right\}\eqsp.
\end{align*}
By the Markov inequality and \cite[Theorem~$4.1$]{lecorff:fort:2011}, for all $p\in (2,\bar p)$, there exists a constant $C$ s.t.
\begin{equation*}
  \PPim\left\{\left|\lyap\circ\bar\param (\mapS(\param_{n}))-\lyap\circ\bar\param (\widetilde S_{n})\right|\geq \varepsilon\right\}\leq \frac{2}{\alpha^{p}}\CE[]{}{}{|\mapS(\param_{n})-\widetilde S_{n}|^{p}} \leq C\tau_{n+1}^{-p/2}\eqsp.
\end{equation*}
\eqref{BOEMsupp:eq:conv1} follows from H\ref{BOEMsupp:assum:size-block} and the Borel-Cantelli lemma (since $p>2$ and $a>1$). 
\end{proof}
Proposition~\ref{BOEMsupp:prop:bonEM:stat} shows that we can address equivalently the convergence of the statistics $\{ \widetilde S_{n}\}_{n\geq 0}$ to some fixed point of $\limEMmap$ and the convergence of the sequence $\{\param_{n}\}_{n\geq 0}$ to some fixed point of $\limEMmapt$.

\begin{proposition}
\label{BOEMsupp:prop:bonEM:stat}
Assume H\ref{BOEMsupp:assum:exp}-\ref{BOEMsupp:assum:strong}, H\ref{BOEMsupp:assum:moment:sup}-($\bar
p$), H\ref{BOEMsupp:assum:obs}\eqref{BOEMsupp:assum:obs:mix}, H\ref{BOEMsupp:assum:size-block} and H\ref{BOEMsupp:assum:SMCapprox}-($\bar p$) for
some $\bar p>2$.
\begin{enumerate}[(i)]
\item \label{BOEMsupp:prop:bonEM:stat:item1} Let $\param_{\star}\in\Staset$.  Set
  $s_\star \eqdef \mapS(\param_\star) = \limEMmap(s_\star)$. Then, $\ps{\PPim}$,
\[  \underset{n\to +\infty}{\lim} \left|\widetilde
  S_{n}-
  s_{\star}\right|\1_{\lim_{n}\param_{n}=\param_{\star}} = 0 \eqsp.
\]
\item \label{BOEMsupp:prop:bonEM:stat:item2} Let $s_\star \in \mathcal{S}$ s.t.
  $\limEMmap(s_\star) = s_\star$.  Set $\param_\star \eqdef \bar \param(s_\star) = \limEMmapt(\param_\star)$. Then $\ps{\PPim}$,
\[
\underset{n\to +\infty}{\lim} \left|\param_n -
  \param_{\star}\right|\1_{\lim_{n} \widetilde
  S_{n} =s_{\star}} = 0 \eqsp.
\]
\end{enumerate}
\end{proposition}
 \begin{proof}
  Let $\mapS$ be given by \eqref{BOEMsupp:eq:mapS}. By \cite[Theorem~$4.1$]{lecorff:fort:2011} and H\ref{BOEMsupp:assum:size-block}, 
\[
\lim_n \left(\widetilde S_{n}-\mapS(\param_n) \right) =0\quad\ps{\PPim}
  \]
  By \cite[Theorem~$4.1$]{lecorff:fort:2011},
$\mapS$ is continuous. Hence, 
\[
\lim_n \left|\widetilde
  S_{n} - \mapS(\param_\star) \right|
\1_{\lim_n \param_n = \param_\star} =0\quad\ps{\PPim}
\] and the proof of
\eqref{BOEMsupp:prop:bonEM:stat:item1} follows. Since $\bar \param$ is continuous,
\eqref{BOEMsupp:prop:bonEM:stat:item2} follows.
 \end{proof}
 
\subsection{Proof of \cite[Proposition~$6.2$]{lecorff:fort:2011}}
\label{BOEMsupp:sec:prop:mu:rho:rate}
We start with rewriting some definitions and assumptions introduced in
\cite{lecorff:fort:2011}.  Define the sequences $\mu_n$ and $\rho_n$, $n \geq 0$
by $\mu_{0} = 0$, $\rho_{0} = \widetilde S_{0}-s_{\star}$ and
\begin{equation}
  \mu_{n} \eqdef \jacG\mu_{n-1} + e_n \eqsp, \qquad \rho_{n} \eqdef
  \widetilde S_{n}-s_{\star} - \mu_{n} \label{BOEMsupp:eq:def:mu} \eqsp,  \qquad n\geq 1 \eqsp,
\end{equation}
where, $\Gamma \eqdef \nabla \limEMmap (s_{\star})$,
\begin{equation}
\label{BOEMsupp:eq:define:epsilon}
e_{n}\eqdef \widetilde S_{n}-\mapS(\param_n)\eqsp, \qquad  n\geq 1 \eqsp,
\end{equation} 
and $\mapS$ is given by (\ref{BOEMsupp:eq:mapS}).

\begin{proof}
Let $p\in (2,\bar p)$. By (\ref{BOEMsupp:eq:def:mu}), for all $n\geq 1$, $\mu_{n}
= \sum_{k=0}^{n-1}\jacG^{k}e_{n-k}$.  By H\ref{BOEMsupp:assum:stable:fixpoint} and
the Minkowski inequality, for all $n\geq 1$, $ \lpnorm{\mu_{n}}{p} \leq
\sum_{k=0}^{n-1}\gamma^{k}\lpnorm{e_{n-k}}{p}$.  By \cite[Theorem~$4.1$]{lecorff:fort:2011}, there exists a
constant $C$ s.t. for any $n\geq 1$,
\[
\lpnorm{\mu_{n}}{p} \leq C\sum_{k=0}^{n-1}\gamma^{k}
\sqrt{\frac{1}{\tau_{n+1-k}}}\eqsp.
\]
By \cite[Result 178, p. 39]{polya:1976} and H\ref{BOEMsupp:assum:size-block}, this yields $\sqrt{\tau_n} \mu_n
= O_{\rmL_{p}}(1)$. 

By H\ref{BOEMsupp:assum:stable:fixpoint}, using a Taylor expansion with integral
form of the remainder term,
\begin{align*}
  \limEMmap(\widetilde S_{n-1})- \limEMmap(s_{\star})&
  -\jacG\left(\widetilde S_{n-1}-s_{\star}\right)
  \\
  &=\sum_{i,j = 1}^{d}\left(\widetilde S_{n-1,i}-s_{\star,i}\right)\left(\widetilde S_{n-1,j}-s_{\star,j}\right)R_{n-1}(i,j)\\
  &=\sum_{i,j =
    1}^{d}(\mu_{n-1,i}+\rho_{n-1,i})(\mu_{n-1,j}+\rho_{n-1,j})R_{n-1}(i,j)\eqsp,
\end{align*}
where $x_{n,i}$ denotes the $i$-th component of $x_{n}\in\Rset^{d}$ and
\[
R_{n}(i,j)\eqdef \int_{0}^{1}(1-t)\frac{\partial^{2} \limEMmap}{\partial
  s_{i}\partial s_{j}}\left(s_{\star}+t(\widetilde S_{n}-s_{\star})\right)\rmd t\eqsp,
\qquad n \in \Nset, 1\leq i,j\leq d \eqsp.
\]
Observe that under H\ref{BOEMsupp:assum:stable:fixpoint}, $\limsup_n |R_n| \1_{\lim_n
  \param_n = \param_\star} < \infty$ w.p.1.  Define for $n \geq 1$ and $k \leq
n$,
\begin{align}
\label{BOEMsupp:eq:def:H-r}
H_{n} &\eqdef \sum_{i=1}^{d}(2\mu_{n,i}+\rho_{n,i})R_{n}(i,\cdot)\eqsp,\quad & r_{n} \eqdef \sum_{i,j=1}^{d}R_{n}(i,j)\mu_{n,i}\mu_{n,j}\eqsp, \\
\psi(n,k) & \eqdef (\Gamma+H_{n})\cdots(\Gamma+H_{k})\eqsp,
\end{align}
with the convention $\psi(n,n+1) \eqdef \mathrm{Id}$.  By \eqref{BOEMsupp:eq:def:mu},
\begin{equation}
\label{BOEMsupp:eq:rho-dec}
\rho_{n} = \psi(n-1,0)\rho_{0} +  \sum_{k=0}^{n-1}\psi(n-1,k+1)r_{k}\eqsp.
\end{equation}
Since $ \sqrt{\tau_n} \mu_n = O_{\rmL_{p}}(1)$, H\ref{BOEMsupp:assum:size-block} and $p>2$ imply that $\mu_{n}\underset{n\to +\infty}{\longrightarrow} 0$, $\ps{\PPim}$ Then, by \eqref{BOEMsupp:eq:def:mu},
$\rho_{n}\1_{\lim_n
  \param_n = \param_\star}\underset{n\to +\infty}{\longrightarrow}
0$, $\ps{\PPim}$ and by \eqref{BOEMsupp:eq:def:H-r} $\underset{n\to
  +\infty}{\lim}\left|H_{n}\right|\1_{\lim_n
  \param_n = \param_\star} = 0$, $\ps{\PPim}$  Let $\widetilde{\gamma}\in (\gamma,1)$, where $\gamma$ is
given by H\ref{BOEMsupp:assum:stable:fixpoint}. Since $\underset{n\to
  +\infty}{\lim}\left|H_{n}\right|\1_{\lim_n
  \param_n = \param_\star} = 0$, there
exists a $\ps{\PPim}$ finite random variable $Z_{1}$ s.t., for all $0\leq
k \leq n-1$,
\begin{equation}
\label{eq:hurwitz}
\left|\psi(n-1,k)\right|\1_{\lim_n
  \param_n = \param_\star}\leq \widetilde{\gamma}^{n-k}Z_{1}\1_{\lim_n
  \param_n = \param_\star}\eqsp.
\end{equation}
Therefore, $\left|\psi(n-1,0)\rho_{0}\right|\1_{\lim_n
  \param_n = \param_\star}\leq
\widetilde{\gamma}^{n} Z_1 \,\left|\rho_{0}\right|\quad \ps{\PPim}$, and, by
H\ref{BOEMsupp:assum:moment:sup}-($\bar p$),
\eqref{BOEMsupp:eq:rewrite:barS}, \eqref{BOEMsupp:eq:define-Phi:theta}, $\CE[]{}{}{|\rho_{0}|^{\bar p}}<+\infty$ which implies that
$\rho_{0}<+\infty$ $\ps{\PPim}$ Since $\widetilde{\gamma}<1$, the first term
in the RHS of (\ref{BOEMsupp:eq:rho-dec}) is $\tau_n^{-1} O_{\rmL_{p}}(1)
O_{\mathrm{a.s}}(1)$.

We now consider the second term in the RHS of \eqref{BOEMsupp:eq:rho-dec}. From equation \eqref{eq:hurwitz}, 
\[
\left|\sum_{k=0}^{n-1}\psi(n-1,k+1)r_{k}\right|\1_{\lim_n
  \param_n = \param_\star}\leq Z_{1}\;\sum_{k=0}^{n-1}\widetilde{\gamma}^{n-k-1}\left|r_{k}\right|\1_{\lim_{n}S_{n}=s_{\star}}\eqsp,\quad\ps{\PPim}
\]
By \eqref{BOEMsupp:eq:def:H-r} and H\ref{BOEMsupp:assum:stable:fixpoint}, there exists a
$\ps{\PPim}$ finite random variable $Z_{2}$ s.t.
\[
\left|r_{k}\right|\1_{\lim_n
  \param_n = \param_\star}\leq
Z_{2}\sum_{i,j=1}^{d}\mu_{k,i}\mu_{k,j}\eqsp,\quad\ps{\PPim}  
\]
In addition, since $
\sqrt{\tau_n} \mu_n = O_{\rmL_{p}}(1)$, there exists a constant $C$ s.t. 
\[
\lpnorm{\sum_{k=0}^{n-1}\widetilde{\gamma}^{n-k-1}\sum_{i,j=1}^{d}\mu_{k,i}\mu_{k,j}}{p/2}\leq
C\;\sum_{k=0}^{n-1}\frac{\widetilde{\gamma}^{n-k-1}}{\tau_{k}}\eqsp.
\]
Applying again \cite[Result 178, p. 39]{polya:1976} yields that the second term in the RHS of
(\ref{BOEMsupp:eq:rho-dec}) is $\tau_n^{-1} O_{\mathrm{a.s}}(1) O_{\rmL_{p/2}}(1)$.
\end{proof}

\section{General results on HMM}
\label{BOEMsupp:sec:forgetting:HMM}
In this section, we derive results on the forgetting properties of HMM
(Section~\ref{BOEMsupp:sec:borw:back}), on their applications to bivariate smoothing distributions (Section~\ref{BOEMsupp:sec:bivariate:smoothing}), 
on the asymptotic behavior of the normalized log-likelihood
(Section~\ref{BOEMsupp:sec:limiting:loglikeli}) and on the normalized score
(Section~\ref{BOEMsupp:sec:normalized:score}).

For any sequence $\bfy\in\Yset^{\Zset}$
and any function $h: \Xset^{2}\times\Yset \to \Rset$, denote by $h_{s}$ the
function on $\Xset^{2} \to \Rset$ given by
\begin{equation}
\label{eq:definition:fonction:hs}
h_{s} (x,x^{\prime}) \eqdef  h(x,x^{\prime},y_{s})\eqsp.
\end{equation}

 \subsection{Forward and Backward forgetting}
\label{BOEMsupp:sec:borw:back} 
In this section, the dependence on $\param$ is dropped from the notation for better clarity. For any $s \in \Zset$ and any $A\in\sigmaX$, define
\begin{equation}\label{eq:kertnel:L}
L_{s}(x,A) \eqdef  \int m(x,x^{\prime})g(x^{\prime},y_{s+1})\1_{A}(x^{\prime})\lambda(\rmd x^{\prime})\eqsp,
\end{equation}
and, for any $s \leq t$ denote by $L_{s:t}$ the composition of the
kernels defined by 
\[  
L_{s:s} \eqdef L_{s} \eqsp, \qquad 
L_{s:u+1}(x,A) \eqdef \int L_{s:u}(x,\rmd
x^{\prime})L_{u+1}(x^{\prime},A)\eqsp.
\]
By convention, $L_{s:s-1}$ is the identity kernel:
$L_{s:s-1}(x,A)=\delta_x(A)$.  For any $\bfy\in\Yset^{\mathbb{Z}}$, any probability distribution $\chi$ on
$\left(\Xset,\sigmaX\right)$ and for any integers such that $r\leq s < t$, let us define two
Markov kernels on $(\Xset,\sigmaX)$ by
\begin{align}
  \rmF_{s,t}(x,A) & \eqdef  \frac{\int L_{s}(x,\rmd x_{s+1})1_A(x_{s+1})L_{s+1:t-1}(x_{s+1},\Xset)}{L_{s:t-1}(x,\Xset)} \eqsp,\label{BOEMsupp:eq:forward-kernel}\\
  \rmB_{s}^{\chi,r}(x,A) & \eqdef \frac{\int \filter{\chi}{r}{s}{r:s}(\rmd
    x_s)1_A(x_s)m(x_{s},x)}{\int \filter{\chi}{r}{s}{r:s}(\rmd
    x_s)m(x_{s},x)}\label{eq:backward-kernel}\eqsp,
\end{align}
where 
\[
\filter{\chi}{r}{s}{r:s}(A) \eqdef \frac{\int \chi(\rmd x_r)L_{r:s-1}(x_r,\rmd
  x_s)1_A(x_s)}{\int \chi(\rmd x_r)L_{r:s-1}(x_r,\Xset)}\eqsp.
\]
Finally, the Dobrushin coefficient of a Markov kernel
$F:\left(\Xset,\sigmaX\right) \longrightarrow [0,1]$ is defined by:
\begin{equation*}
\delta(F) \eqdef \frac{1}{2}\underset{\left(x,x^{\prime}\right)\in\Xset^2}{\sup}\left|\left|F(x,\cdot) - F(x^{\prime},\cdot)\right|\right|_{\mathrm{TV}}\eqsp.
\end{equation*}

\begin{lemma}\label{BOEMsupp:prop:smoothing-kernels}
  Assume that there exist positive numbers $\sigma_-, \sigma_+$ such that
  $\sigma_- \leq m(x,x') \leq \sigma_+$ for any $x,x' \in \Xset$. Then for any
  $\bfy \in \Yset^\Zset$, $\delta(\rmF_{s,t}) \leq \rho$ and
  $\delta(\rmB_{s}^{\chi,r}) \leq \rho$ where $\rho \eqdef \sigma_-/\sigma_+$.
\end{lemma}
\begin{proof}
  Let $r$, $s$, $t$ be such that $r\leq s < t$. Under the stated assumptions, 
\begin{multline*}
  \int L_{s}(x_s,\rmd x_{s+1})1_A(x_{s+1})L_{s+1:t-1}(x_{s+1},\Xset) \\
  \geq\sigma_-\int g(x_{s+1},y_{s+1})1_A(x_{s+1})L_{s+1:t-1}(x_{s+1},\Xset)\lambda(\rmd x_{s+1})
\end{multline*}
and
\[
L_{s:t-1}(x_{s},\Xset)\leq\sigma_+\int
g(x_{s+1},y_{s+1})L_{s+1:t-1}(x_{s+1},\Xset)\lambda(\rmd x_{s+1})\eqsp.
\]
This yields to
\begin{equation*}
\rmF_{s,t}(x_s,A) \geq \frac{\sigma_-}{\sigma_+}\frac{\int g(x_{s+1},y_{s+1})L_{s+1:t-1}(x_{s+1},\Xset)\1_A(x_{s+1})\lambda(\rmd x_{s+1})}{\int g(x_{s+1},y_{s+1})L_{s+1:t-1}(x_{s+1},\Xset)\lambda(\rmd x_{s+1})}\eqsp.
\end{equation*}
Similarly, the assumption implies
\begin{equation*}
\rmB_{s}^{\chi,r}(x_{s+1},A) \geq \frac{\sigma_-}{\sigma_+}\filter{\chi}{r}{s}{r:s}(A)\eqsp,
\end{equation*}
which gives the upper bound for the Dobrushin coefficients, see
\cite[Lemma 4.3.13]{cappe:moulines:ryden:2005}.
\end{proof}

\begin{lemma}\label{BOEMsupp:lem:back-forget}
  Assume that there exist positive numbers $\sigma_-, \sigma_+$ such that
  $\sigma_- \leq m(x,x') \leq \sigma_+$ for any $x,x' \in \Xset$. Let $\bfy \in
  \Yset^\Zset$. 
  \begin{enumerate}[(i)]
  \item \label{BOEMsupp:eq:forget-chi} for any bounded function $h$, any probability distributions $\chi$ and
    $\widetilde{\chi}$ and any integers $r\leq s \leq t$
\begin{multline}
  \left |\frac{\int  \chi(\rmd x_r)L_{r:s-1}(x_r,\rmd x_s)h(x_s)L_{s:t-1}(x_s,\Xset)}{\int  \chi(\rmd x_r)L_{r:t-1}(x_r,\Xset)}\right.\\
  \hspace{1cm}\left.- \frac{\int \widetilde{\chi}(\rmd
      x_r)L_{r:s-1}(x_r,\rmd x_s)h(x_s)L_{s:t-1}(x_s,\Xset)}{\int
      \widetilde{\chi}(\rmd x_r)L_{r:t-1}(x_r,\Xset)} \right |
  \leq \rho^{s-r}\osc(h)\eqsp,
\end{multline}
\item \label{BOEMsupp:eq:forget-f} for any bounded function $h$, for any non-negative functions $f$ and
  $\widetilde{f}$ and any integers $r\leq s \leq t$
\begin{multline}
  \left | \frac{\int \chi(\rmd x_s)h(x_s)L_{s:t-1}(x_s,\rmd x_t)f(x_t)}{\int \chi(\rmd x_s)L_{s:t-1}(x_s,\rmd x_t)f(x_t)}\right. \\
  \hspace{3cm}\left.- \frac{\int \chi(\rmd x_s)h(x_s)L_{s:t-1}(x_s,\rmd
      x_t)\widetilde{f}(x_t)}{\int \chi(\rmd x_s)L_{s:t-1}(x_s,\rmd
      x_t)\widetilde{f}(x_t)} \right | \leq \rho^{t-s}\osc(h)\eqsp.
\end{multline}  \end{enumerate}
\end{lemma}

\begin{proof}[\textit{Proof of \eqref{BOEMsupp:eq:forget-chi}}]
See \cite[Proposition 4.3.23]{cappe:moulines:ryden:2005}.

\textit{Proof of \eqref{BOEMsupp:eq:forget-f}}  
When $s=t$, then \eqref{BOEMsupp:eq:forget-f} is equal to
\[
\left | \frac{\int \chi(\rmd x_t)h(x_t)f(x_t)}{\int \chi(\rmd x_t)f(x_t)}-
  \frac{\int \chi(\rmd x_t)h(x_t)\widetilde{f}(x_t)}{\int \chi(\rmd
    x_t)\widetilde{f}(x_t)} \right | \eqsp.
\]
This is of the form $\left(\eta - \widetilde{\eta}\right)h$ where $\eta$ and
$\widetilde{\eta}$ are probability distributions on
$\left(\Xset,\sigmaX\right)$. Then,
 \[\left|\left(\eta - \widetilde{\eta}\right)h\right|\leq \frac{1}{2}\normTV{\eta - \widetilde{\eta}}\osc\left(h\right) \leq \osc(h)\eqsp.\]

 Let $s < t$. By definition of the backward smoothing kernel, see
 \eqref{eq:backward-kernel},
\begin{equation*}
\rmB_{s}^{\chi,s}(x_{s+1},A) = \frac{\int \chi(\rmd x_s)1_A(x_s)m(x_s,x_{s+1})}{\int \chi(\rmd x_s)m(x_s,x_{s+1})}\eqsp.
\end{equation*}
Therefore,
\begin{multline*}
  \int \chi(\rmd x_s)h(x_s)L_{s:t-1}(x_s,\rmd x_t)f(x_t)\\
  = \int \chi(\rmd x_s)L_{s}(x_s,\rmd
  x_{s+1})\rmB_{s}^{\chi,s}h(x_{s+1})L_{s+1:t-1}(x_{s+1},\rmd x_t)f(x_t)\eqsp.
\end{multline*}
By repeated application of the backward smoothing kernel we have
\begin{multline*}
\int \chi(\rmd x_s)h(x_s)L_{s:t-1}(x_s,\rmd x_t)f(x_t)\\
= \int \chi(\rmd x_s)L_{s:t-1}(x_s,\rmd x_{t})\rmB_{t-1:s}^{\chi,s}h(x_{t})f(x_t)\eqsp,
\end{multline*}
where we denote by $\rmB_{t-1:s}^{\chi,s}$ the composition of the kernels
defined by induction for $s \leq u$
\[
\rmB_{s:s}^{\chi,s} \eqdef \rmB_{s}^{\chi,s} \eqsp, \qquad
\rmB_{u:s}^{\chi,s}(x,A) \eqdef \int \rmB_{u}^{\chi,s}(x,\rmd
x^{\prime})\rmB_{u-1:s}^{\chi,s}(x^{\prime},A)\eqsp.
\]
Finally, by definition of $\filter{\chi}{s}{t}{s:t}$
\begin{multline*}
  \left | \frac{\int \chi(\rmd x_s)h(x_s)L_{s:t-1}(x_s,\rmd x_t)f(x_t)}{\int \chi(\rmd x_s)L_{s:t-1}(x_s,\rmd x_t)f(x_t)}- \frac{\int \chi(\rmd x_s)h(x_s)L_{s:t-1}(x_s,\rmd x_t)\widetilde{f}(x_t)}{\int \chi(\rmd x_s)L_{s:t-1}(x_s,\rmd x_t)\widetilde{f}(x_t)} \right | \\
  = \left |
    \frac{\filter{\chi}{s}{t}{s:t}\left[\left(\rmB_{t-1:s}^{\chi,s}h\right)f\right]}{\filter{\chi}{s}{t}{s:t}\left[f\right]}
    -
    \frac{\filter{\chi}{s}{t}{s:t}\left[\left(\rmB_{t-1:s}^{\chi,s}h\right)\widetilde{f}\right]}{\filter{\chi}{s}{t}{s:t}\left[\widetilde{f}\right]}\right
  |\eqsp.
 \end{multline*}
 This is of the form $\left(\eta -
   \widetilde{\eta}\right)\rmB_{t-1:s}^{\chi,s}h$ where $\eta$ and
 $\widetilde{\eta}$ are probability distributions on
 $\left(\Xset,\sigmaX\right)$. The proof of the second statement is
 completed upon noting that
\begin{align*}
  \left|\eta\rmB_{t-1:s}^{\chi,s}h - \widetilde{\eta}\rmB_{t-1:s}^{\chi,s}h\right|&\leq \frac{1}{2}\normTV{\eta - \widetilde{\eta}}\osc\left(\rmB_{t-1:s}^{\chi,s}h\right)\\
  &\leq \frac{1}{2}\normTV{\mu - \widetilde{\mu}}\delta\left(\rmB_{t-1:s}^{\chi,s}\right)\osc(h) \leq \rho^{t-s}\osc(h)\eqsp,\\
\end{align*}
where we used Lemma~\ref{BOEMsupp:prop:smoothing-kernels} in the last inequality.
\end{proof}

\subsection{Bivariate smoothing distribution}
\label{BOEMsupp:sec:bivariate:smoothing}
\begin{proposition}\label{BOEMsupp:prop:expforget}
  Assume H\ref{BOEMsupp:assum:strong}. Let $\chi$, $\widetilde{\chi}$ be two
  distributions on $\left(\Xset,\sigmaX\right)$.  For any measurable
  function $h: \Xset^{2}\times \Yset \to \Rset^d$ and any $\bfy \in
  \Yset^\Zset$ such that $\sup_{x,x'}|h(x,x^{\prime},y_s)| < +\infty$ for any $s
  \in \Zset$
\begin{enumerate}[(i)]
\item \label{item:forget-nu-pv} For any $r < s\leq t$ and any $\ell_{1},
  \ell_2\geq 1$,
\begin{equation}\label{BOEMsupp:eq:forget-nu-pv}
\underset{\param\in\paramset}{\sup}\left | \smoothfunc{\widetilde{\chi},r}{\param,s,t}\left(h,\bfy\right) - \smoothfunc{\chi,r-\ell_{1}}{\param,s,t+\ell_{2}}\left(h,\bfy\right) \right | \leq \left (\rho^{s-1-r} + \rho^{t-s}\right )\osc(h_{s})\eqsp.
\end{equation}
\item \label{item:forget-limit} For any $\param\in\paramset$, there exists a
  function $\bfy \mapsto \Phi_{\param}(h,\bfy)$ s.t. for any distribution
  $\chi$ on $(\Xset,\sigmaX)$ and any $r<s\leq t$
\begin{equation}\label{sup:eq:forget-limit}
\underset{\param\in\paramset}{\sup}\left|\smoothfunc{\chi,r}{\param,s,t}\left(h,\bfy\right) - \smoothfunc{}{\param}\left(h,\shift^s\bfy\right)\right|\leq \left (\rho^{s-1-r} + \rho^{t-s}\right )\osc(h_{s})\eqsp.
\end{equation}
\end{enumerate}
\end{proposition}
\begin{remark}\label{BOEMsupp:rem:lgn}
\begin{enumerate}[(a)]
\item If $\chi=\widetilde{\chi}$, $\ell_{1}=0$ and $\ell_{2}\geq 1$,
  \eqref{BOEMsupp:eq:forget-nu-pv} becomes
\[
\underset{\param\in\paramset}{\sup}\left | \smoothfunc{\chi,r}{\param,s,t}\left(h,\bfy\right) - \smoothfunc{\chi,r}{\param,s,t+\ell_{2}}\left(h,\bfy\right) \right | \leq \rho^{t-s}\osc(h_{s})\eqsp.
\]
\item if $\ell_{2} = 0$ and $\ell_{1}\geq 1$, \eqref{BOEMsupp:eq:forget-nu-pv} becomes
\[
\underset{\param\in\paramset}{\sup}\left | \smoothfunc{\widetilde{\chi},r}{\param,s,t}\left(h,\bfy\right) - \smoothfunc{\chi,r-\ell_{1}}{\param,s,t}\left(h,\bfy\right) \right | \leq \rho^{s-1-r}\osc(h_{s})\eqsp.
\]
\end{enumerate}
\end{remark}
\begin{proof} 
  \textit{(i)} Let $r$, $s$, $t$ such that $r < s\leq t$, $\ell_{1},\ell_2 \geq
  1$, and $\param\in\paramset$. Define the distribution
  $\chi_{\param,r-\ell_{1}:r}$ on $(\Xset,\sigmaX)$ by
\begin{equation*}
\chi_{\param,r-\ell_{1}:r}(A) \eqdef \frac{\int\chi(\rmd x_{r-\ell_{1}})L_{\param,r-\ell_{1}:r-1}(x_{r-\ell_{1}},\rmd x_{r})1_{A}(x_r)}{\int\chi(\rmd x_{r-\ell_{1}})L_{\param,r-\ell_{1}:r-1}(x_{r-\ell_{1}},\Xset)}\eqsp, \quad \forall A\in\sigmaX\eqsp.
\end{equation*}
We write $\left |
  \smoothfunc{\widetilde{\chi},r}{\param,s,t}\left(h,\bfy\right) -
  \smoothfunc{\chi,r-\ell_{1}}{\param,s,t+\ell_{2}}\left(h,\bfy\right) \right |
\leq \widetilde{T}_1 + \widetilde{T}_2$ where, by using \eqref{BOEMsupp:eq:define-Phi:theta},
\begin{multline*}
  \widetilde{T}_1 \eqdef  \left | \frac{\int \widetilde{\chi}(\rmd x_r)L_{\param,r:s-2}(x_r,\rmd x_{s-1})h_{s}(x_{s-1},x_s)L_{\param,s-1}(x_{s-1},\rmd x_s)L_{\param,s:t-1}(x_s,\Xset)}{\int \widetilde{\chi}_{r}(\rmd x_r)L_{\param,r:t-1}(x_r,\Xset)}\right.\\
  \left. - \frac{\int \chi_{\param,r-\ell_{1}:r}(\rmd x_r)L_{\param,r:s-2}(x_r,\rmd
      x_{s-1})h_{s}(x_{s-1},x_s)L_{\param,s-1}(x_{s-1},\rmd
      x_s)L_{\param,s:t-1}(x_s,\Xset)}{\int \chi_{\param,r-\ell_{1}:r}(\rmd
      x_r)L_{\param,r:t-1}(x_r,\Xset)}\right |\eqsp,
\end{multline*}
and
\begin{multline*}
  \widetilde{T}_2 \eqdef \left | \frac{\int \chi_{\param,r-\ell_{1}:r}(\rmd
      x_r)L_{\param,r:s-2}(x_r,\rmd
      x_{s-1})h_{s}(x_{s-1},x_s)L_{\param,s-1}(x_{s-1},\rmd
      x_s)L_{\param,s:t-1}(x_s,\Xset)}{\int \chi_{\param,r-\ell_{1}:r}(\rmd
      x_r)L_{\param,r:t-1}(x_r,\Xset)}\right.\\
  \left. -\frac{\int \chi_{\param,r-\ell_{1}:r}(\rmd
      x_r)L_{\param,r:s-2}(x_r,\rmd
      x_{s-1})h_{s}(x_{s-1},x_s)L_{\param,s-1}(x_{s-1},\rmd
      x_s)L_{\param,s:t+\ell_2-1}(x_s,\Xset)}{\int
      \chi_{\param,r-\ell_{1}:r}(\rmd x_r)L_{\param,r:t+\ell_{2}-1}(x_r,\Xset)}\right|\eqsp.
\end{multline*}
Set $\bar h_{s,t} : x \mapsto \int F_{\param,s-1,t}(x,\rmd x_s)h_{s}(x,x_s)$
where $F_{\param,s-1,t}$ is the forward smoothing kernel (see \eqref{BOEMsupp:eq:forward-kernel}). Then,
\begin{multline*}
  \widetilde{T}_1=\left|\frac{\int \widetilde{\chi}(\rmd x_r)L_{\param,r:s-2}(x_r,\rmd x_{s-1})\bar{h}_{s,t}(x_{s-1})L_{\param,s-1:t-1}(x_{s-1},\Xset)}{\int \widetilde{\chi}_{r}(\rmd x_r)L_{\param,r:t-1}(x_r,\Xset)}\right.\\
  -\left. \frac{\int \chi_{\param,r-\ell_{1}:r}(\rmd x_r)L_{\param,r:s-2}(x_r,\rmd
      x_{s-1})\bar{h}_{s,t}(x_{s-1})L_{\param,s-1:t-1}(x_{s-1},\Xset)}{\int \chi_{\param,r-\ell_{1}:r}(\rmd
      x_r)L_{\param,r:t-1}(x_r,\Xset)}\right|\eqsp.
\end{multline*}
By Lemma~\ref{BOEMsupp:lem:back-forget}\eqref{BOEMsupp:eq:forget-chi},
\begin{equation*}
\widetilde{T}_1 \leq \rho^{s-1-r}\osc(\bar{h}_{s,t})\leq 2\rho^{s-1-r}\sup_{x\in\Xset}|\bar h_{s,t}(x)| \leq 2\rho^{s-1-r}\sup_{(x,x^{\prime})\in\Xset^{2}}|h_{s}(x,x^{\prime})|\eqsp.
\end{equation*}
Set $\widetilde{h}_{s}: x \mapsto \int
\rmB_{\param,s-1}^{\chi_{\param,r-\ell_{1}:s-1},s-1}(x,\rmd
x_{s-1})h_{s}(x_{s-1},x)\eqsp,$ where
$\rmB_{\param,s-1}^{\chi_{\param,r-\ell_{1}:s-1},s-1}$ is the backward
smoothing kernel (see \eqref{eq:backward-kernel}).  Then,
\begin{multline*}
  \widetilde{T}_2 = \left|\frac{\int \chi_{\param,r-\ell_{1}:s}(\rmd x_{s})\widetilde{h}_s(x_{s})L_{\param,s:t-1}(x_{s},\rmd x_t)L_{\param,t:t+\ell_{2}-1}(x_t,\Xset)}{\int \chi_{\param,r-\ell_{1}:s}(\rmd x_{s})L_{\param,s:t-1}(x_{s},\rmd x_t)L_{\param,t:t+\ell_{2}-1}(x_t,\Xset)}  \right.\\
    - \left. \frac{\int \chi_{\param,r-\ell_{1}:s}(\rmd
      x_{s})\widetilde{h}_s(x_{s})L_{\param,s:t-1}(x_{s},\Xset)}{\int
      \chi_{\param,r-\ell_{1}:s}(\rmd
      x_{s})L_{\param,s:t-1}(x_{s},\Xset)}\right|\eqsp.
\end{multline*}
Then, by Lemma~\ref{BOEMsupp:lem:back-forget}\eqref{BOEMsupp:eq:forget-f},
\begin{equation*}
\widetilde{T}_2 \leq \rho^{t-s}\osc(\widetilde{h}_s)\leq 2\rho^{t-s}\sup_{x\in\Xset}|\widetilde{h}_{s}(x)|\leq 2\rho^{t-s}\sup_{(x,x^{\prime})\in\Xset^{2}}|h_{s}(x,x^{\prime})|\eqsp.
\end{equation*}
The proof is concluded upon noting that, for any constant $c$, 
\[
\osc(h) = 2\underset{c\in\Rset}{\inf}\left\{\sup_{(x,x^{\prime})\in\Xset^{2}}|h_{s}(x,x^{\prime})-c|\right\}\eqsp.
\]

\textit{(ii)} By \eqref{BOEMsupp:eq:forget-nu-pv}, for any increasing sequence of non
negative integers $(r_{\ell})_{\ell\geq 0}$, $(t_{\ell})_{\ell\geq 0}$ s.t.
$\lim r_\ell = \lim t_\ell = +\infty$, the sequence
$\{\smoothfunc{\chi,-r_{\ell}}{\param,0,t_{\ell}}\left(h,\bfy\right)\}_{\ell
  \geq 0}$ is a Cauchy sequence uniformly in $\param$ and $\chi$. Then, there exists
a limit $\smoothfunc{}{\param}\left(h,\bfy\right)$ s.t.
\begin{equation}
  \label{eq:prop:forget:claim2:tool1}
  \lim_{\ell \to +\infty}    \sup_{\chi} \underset{\param\in\paramset}{\sup}\left|\smoothfunc{\chi,-r_{\ell}}{\param,0,t_{\ell}}\left(h,\bfy\right) - \smoothfunc{}{\param}\left(h,\bfy\right)\right| =0 \eqsp.
\end{equation}
We write, for any $r< s\leq t$ and any $\ell\geq 1$
\begin{multline*}
  \left| \smoothfunc{\chi,r}{\param,s,t}\left(h,\bfy\right) -
    \smoothfunc{}{\param}\left(h,\shift^s\bfy\right)
  \right|  \\
  \leq \left| \Phi^{\chi,r}_{\param,s,t}\left(h,\bfy\right) -
    \smoothfunc{\chi,r-\ell}{\param,s,t+\ell}\left(h,\bfy\right) \right| +
  \left|\smoothfunc{\chi,r-\ell}{\param,s,t+\ell}\left(h,\bfy\right) -
    \smoothfunc{}{\param}\left(h,\shift^s\bfy\right) \right| \eqsp.
\end{multline*}
Since $\smoothfunc{\chi,r-\ell}{\param,s,t+\ell}\left(h,\bfy\right)
=\smoothfunc{\chi,r-\ell-s}{\param,0,t+\ell-s}\left(h,\shift^s\bfy\right)$,
Proposition~\ref{BOEMsupp:prop:expforget}\eqref{item:forget-nu-pv} yields
\begin{multline*}
  \left|\smoothfunc{\chi,r}{\param,s,t}\left(h,\bfy\right) - \smoothfunc{}{\param}\left(h,\shift^s\bfy\right)\right| \leq \left (\rho^{s-r-1} + \rho^{t-s}\right )\osc(h_{s})  \\
  +
  \left|\smoothfunc{\chi,r-\ell-s}{\param,0,t+\ell-s}\left(h,\shift^s\bfy\right)
    - \smoothfunc{}{\param}\left(h,\shift^s\bfy\right) \right| \eqsp.
\end{multline*}
The proof is concluded by \eqref{eq:prop:forget:claim2:tool1}.
\end{proof}

\subsection{Limiting normalized log-likelihood}
\label{BOEMsupp:sec:limiting:loglikeli} 
This section contains results adapted from \cite{douc:moulines:ryden:2004} which are stated here for better clarity.
Define for any $r\leq s$,
\begin{equation}
  \label{BOEMsupp:eq:delta-loglikelihood}
   \delta_{\param,s}^{\chi,r}(\bfy) \eqdef \ell_{\param,s+1}^{\chi,r}(\bfy) - \ell_{\param,s}^{\chi,r}(\bfy) \eqsp,
\end{equation}
where $\ell_{\param,s+1}^{\chi,r}(\bfy)$ is defined by
\begin{equation}
\label{BOEMsupp:eq:def:loglikelihood}
\loglik{\chi,r}{\param,s+1}(\bfY) \eqdef \log\int \chi(\rmd x_r)
  \prod_{u=r+1}^{s+1} m_\param(x_{u-1},x_u) g_\param(x_u, \bfY_u) \ \lambda(\rmd x_{r+1:s+1})\eqsp.
\end{equation}
For any $T>0$ and any probability distribution $\chi$ on
$(\Xset,\sigmaX)$, we thus have
\begin{equation}
  \label{BOEMsupp:eq:loglikelihood-delta}
  \ell^{\chi,0}_{\param,T}(\bfy) = \sum_{s=0}^{T-1} \left( \ell^{\chi,0}_{\param,s+1}(\bfy) -
  \ell^{\chi,0}_{\param,s}(\bfy) \right) = \sum_{s=0}^{T-1}\delta_{\param,s}^{\chi,0}(\bfy) \eqsp.
\end{equation}
It is established in Lemma~\ref{BOEMsupp:lem:forget-log-lik} that for any $\param \in
\paramset$, $\bfy \in \Yset^\Zset$, $s \geq 0$ and any initial distribution
$\chi$, the sequence $\{\delta_{\param,s}^{\chi,s-r}(\bfy) \}_{r \geq 0}$ is a
Cauchy sequence and its limit does not depend upon $\chi$. Regularity
conditions on this limit are given in Lemmas~\ref{BOEMsupp:lem:bound-log-lik} and
\ref{BOEMsupp:lem:cont-log-lik}. Finally, Theorem~\ref{BOEMsupp:th:likelihood} shows that for any
$\param$, $\lim_T T^{-1} \ell_{\param,T}^{\chi,0}(\bfY)$ exists w.p.1. and this
limit is a (deterministic) continuous function in $\param$.
\begin{lemma}\label{BOEMsupp:lem:forget-log-lik}
  Assume H\ref{BOEMsupp:assum:strong}.
    \begin{enumerate}[(i)]
  \item \label{BOEMsupp:lem:forget-log-lik:claim1} For any $\ell,r,s \geq 0$, any
    initial distributions $\chi, \chi'$ on $\Xset$ and any $\bfy \in
    \Yset^\Zset$
\begin{equation*}
\underset{\param\in\paramset}{\sup}\left|\delta_{\param,s}^{\chi,s-r}(\bfy)-\delta_{\param,s}^{\chi',s-r-\ell}(\bfy)\right|\leq \frac{2 }{1-\rho}\rho^{r }\eqsp.
\end{equation*}
\item \label{BOEMsupp:lem:forget-log-lik:claim2} For any $\param \in \paramset$, there
  exists a function $\bfy \mapsto \delta_{\param}(\bfy)$ such that for any
  initial distribution $\chi$, any $\bfy \in \Yset^\Zset$ and any $r,s\geq 0$,
\[
\underset{\param\in\paramset}{\sup}\left|\delta_{\param,s}^{\chi,s-r}(\bfy)-\delta_{\param}(\shift^s\circ\bfy
  )\right|\leq \frac{2}{1-\rho}\rho^{r}\eqsp.
\]
  \end{enumerate}
\end{lemma}
\begin{proof}
  {\em Proof of~\eqref{BOEMsupp:lem:forget-log-lik:claim1}.} Let $s\geq 0$ and $r$ and $r^{\prime}$ be such that $r^{\prime}>r$. By
\eqref{BOEMsupp:eq:delta-loglikelihood}   and \eqref{BOEMsupp:eq:def:loglikelihood}, we have
  $|\delta_{\param,s}^{\chi,s-r}(\bfy)-\delta_{\param,s}^{\chi',s-r'}(\bfy)| =
  |\log \alpha - \log \beta|$ where
  \begin{align}\label{eq:defalpha}
    \alpha  & \eqdef  \frac{\int \chi(\rmd x_{s-r}) \prod_{i=s-r+1}^{s+1} m_{\param}(x_{i-1},x_i) g_\param(x_i, y_i) \lambda(\rmd x_i) }{\int \chi(\rmd x_{s-r})\prod_{i=s-r+1}^{s} m_{\param}(x_{i-1},x_i) g_\param(x_i, y_i) \lambda(\rmd x_i)}  \eqsp, \\
    \beta & \eqdef \frac{\int \chi'(\rmd x_{s-r'})\prod_{i=s-r'+1}^{s+1 }m_{\param}(x_{i-1},x_i) g_\param(x_i, y_i) \lambda(\rmd
      x_i) }{\int \chi'(\rmd x_{s-r'}) 
      \prod_{i=s-r'+1}^{s} m_{\param}(x_{i-1},x_i) g_\param(x_i, y_i) \lambda(\rmd x_i)}
    \nonumber\eqsp.
  \end{align}
  We prove that
 \begin{eqnarray}\label{eq:ineqalphabeta}
 \alpha \wedge \beta \geq \sigma_-\int g_{\param}(x_{s+1},y_{s+1})\lambda(\rmd x_{s+1})
 \eqsp, 
\\
 |\alpha - \beta| \leq 2 \rho^{r} \sigma_+\int
 g_{\param}(x_{s+1},y_{s+1})\lambda(\rmd x_{s+1}) \eqsp,
\end{eqnarray}
and the proof is concluded since $|\log \alpha - \log \beta | \leq |\alpha
-\beta| / (\alpha \wedge \beta)$.
  
The minorization on $\alpha$ and $\beta$ is a consequence of
H\ref{BOEMsupp:assum:strong} upon noting that $\alpha$ and $\beta$ are of the form $\int
\mu( \rmd x_{s}) m_\param(x_{s},x_{s+1}) g_\param(x_{s+1}, y_{s+1})
\lambda(\rmd x_{s+1})$ for some probability measure $\mu$. The upper bound on
$|\alpha - \beta|$ is a consequence of
Lemma~\ref{BOEMsupp:lem:back-forget}\eqref{BOEMsupp:lem:forget-log-lik:claim1} applied with
\[
\tilde \chi(\rmd x_{s-r}) \leftarrow \int_{\Xset^{r'-r}} \chi'(\rmd x_{s-r'})
\left\{ \prod_{i=s-r'}^{s-r-1} g_\param(x_{i},y_{i}) m_{\param}(x_{i},x_{i+1})
   \right\}\ \lambda(\rmd x_{s-r'+1:s-r})
\]
and $h(u) \leftarrow \int g_\param(x_{s+1}, y_{s+1}) m_{\param}(u,x_{s+1}) \lambda(\rmd
x_{s+1})$.

{\em Proof of~\eqref{BOEMsupp:lem:forget-log-lik:claim2}.} By
\eqref{BOEMsupp:lem:forget-log-lik:claim1}, for any $\bfy \in \Yset^\Zset$, the sequence
$\{\delta_{\param,0}^{\chi,-r}(\bfy) \}_{r \geq 0}$ is a Cauchy sequence uniformly in $\param$: there
exists a limit denoted by $\delta_\param(\bfy )$ - which does not depend upon
$\chi$ - such that
\begin{equation}
  \label{eq:BOEMsupp:lem:forget-log-lik:claim2:tool1}
  \lim_{r \to +\infty} \sup_{\param \in \paramset} \left|
  \delta_{\param,0}^{\chi,-r}(\bfy) - \delta_\param(\bfy)\right| =0 \eqsp.
\end{equation}
We write for $r \leq r'$
\[
\left| \delta_{\param,s}^{\chi,s-r}(\bfy) - \delta_\param(\shift^s\circ\bfy)
\right| \leq \left| \delta_{\param,s}^{\chi,s-r}(\bfy) -
  \delta_{\param,s}^{\chi,s-r'}(\bfy) \right| +
\left|\delta_{\param,s}^{\chi,s-r'}(\bfy) - \delta_\param(\shift^s\circ\bfy)
\right| \eqsp.
\]
Observe that by definition, $\delta_{\param,s}^{\chi,s-r}(\bfy) =
\delta_{\param,0}^{\chi,-r}(\shift^s\circ\bfy)$. This property, combined with
Lemma~\ref{BOEMsupp:lem:forget-log-lik}\eqref{BOEMsupp:lem:forget-log-lik:claim1}, yield
\[
\sup_{\param \in \paramset} \left| \delta_{\param,s}^{\chi,s-r}(\bfy) -
  \delta_\param(\shift^s\circ\bfy) \right| \leq \frac{2}{1-\rho} \rho^{r } +
\sup_{\param \in \paramset}
\left|\delta_{\param,0}^{\chi,-r'}(\shift^s\circ\bfy) -
  \delta_\param(\shift^s\circ\bfy) \right| \eqsp.
\]
When $r' \to +\infty$, the second term in the rhs tends to zero by
(\ref{eq:BOEMsupp:lem:forget-log-lik:claim2:tool1}) - for fixed $\bfy, s$ and $\chi$ -.
This concludes the proof.
\end{proof}

\begin{lemma}\label{BOEMsupp:lem:bound-log-lik}
  Assume H\ref{BOEMsupp:assum:strong}.
  For any $\bfy \in \Yset^\Zset$ and $s \geq 0$,
\[
\sup_{r\geq 0} \sup_{\param \in \paramset}
\left|\delta_{\param,s}^{\chi,s-r}(\bfy) \right| \leq \left|\log \sigma_+
  b_+(\bfy_{s+1})\right| + \left|\log \sigma_- b_-(\bfy_{s+1})\right| \eqsp,
\]
 and, for any $r\geq 0$,
\[
\sup_{\param \in \paramset} |\delta_{\param}(\bfy)|\leq \frac{2}{(1-\rho)}\rho^r + 
\left|\log \sigma_+  b_+(\bfy_{1})\right|+\left|\log \sigma_- b_-(\bfy_1)\right| \eqsp,
\]where $b_+$ and $b_-$ are defined by (\ref{BOEMsupp:eq:b-}) and (\ref{BOEMsupp:eq:b+}) .
\end{lemma}
\begin{proof}
  For any $ 0< m\leq A/B \leq M$, $|\log(A/B)| \leq \left|\log M\right|+\left|
    \log m\right|$. Note that by definition, $\delta_{\param,s}^{\chi,0}(\bfy)$
  is of the form $\log (A/B)$ and under
  H\ref{BOEMsupp:assum:obs}\eqref{BOEMsupp:assum:obs:b+:b-}, $\sigma_-b_-(y_{s+1}) \leq A/B \leq
  \sigma_+ b_+(y_{s+1}) $. The second upper bound is a consequence of
  Lemma~\ref{BOEMsupp:lem:forget-log-lik}\eqref{BOEMsupp:lem:forget-log-lik:claim2}.
\end{proof}

\begin{lemma}\label{BOEMsupp:lem:cont-log-lik}
  Assume H\ref{BOEMsupp:assum:exp}-\ref{BOEMsupp:assum:strong} and
  H\ref{BOEMsupp:assum:obs}. Then, $\param \mapsto
  \CE[]{}{}{\delta_\param(\bfY)}$ is continuous on $\paramset$ and
\begin{equation}
\label{BOEMsupp:lem:cont-log-lik:moment}
\underset{\eta\rightarrow 0}{\lim}\,\Expparam{\underset{\{\param,\param' \in \paramset;   \left|\param-\param^{\prime}\right|<\eta \}}{\sup}\left|\delta_{\param}(\bfY)-\delta_{\param'}(\bfY)\right|}{} = 0\eqsp,\quad\ps{\PPim}
\end{equation}
\end{lemma}
\begin{proof} 
  By the dominated convergence theorem, Lemma~\ref{BOEMsupp:lem:bound-log-lik} and
  H\ref{BOEMsupp:assum:obs}\eqref{BOEMsupp:assum:obs:b+:b-}, $\param \mapsto
  \Expparam{\delta_\param(\bfY)}{\star}$ is continuous if $\param \mapsto
  \delta_\param(\bfy)$ is continuous for any $\bfy \in \Yset^\Zset$. Let $\bfy
  \in \Yset^\Zset$.  By
  Lemma~\ref{BOEMsupp:lem:forget-log-lik}\eqref{BOEMsupp:lem:forget-log-lik:claim2}, $\lim_{r \to
    +\infty} \sup_{\param \in \paramset} |\delta_{\param,0}^{\chi,-r}(\bfy) -
  \delta_\param(\bfy)| =0$. Therefore, $\param \mapsto \delta_\param(\bfy)$ is
  continuous provided for any $r \geq 0$, $\param \mapsto
  \delta_{\param,0}^{\chi,-r}(\bfy)$ is continuous (for fixed $\bfy$ and
  $\chi$). By definition of $\delta_{\param,0}^{\chi,-r}(\bfy)$, see
  \eqref{BOEMsupp:eq:delta-loglikelihood}, it is sufficient to prove that $\param
  \mapsto \ell_{\param,s}^{\chi,-r}(\bfy)$ is continuous for $s \in \{0, 1 \}$.
  By definition of $\ell_{\param,s}^{\chi,-r}(\bfy)$,  see \eqref{BOEMsupp:eq:def:loglikelihood},
\[
\ell_{\param,s}^{\chi,-r}(\bfy) = \log \int \chi(\rmd x_{-r})\prod_{i=-r+1}^s m_{\param}(x_{i-1},x_i) g_\param(x_i, y_{i})
\lambda(\rmd x_i) \eqsp.
\]
Under H\ref{BOEMsupp:assum:exp}\eqref{BOEMsupp:assum:exp:decomp}, $\param \mapsto\prod_{i=-r+1}^s
m_{\param}(x_{i-1},x_i) g_\param(x_i, y_{i})$ is continuous on $\paramset$, for
any $x_{-r:s}$ and $\bfy$. 
In addition, under H\ref{BOEMsupp:assum:exp}, for any $\param\in\paramset$,
\begin{multline*}
\left|\prod_{i=-r+1}^{s}m_{\param}(x_i, x_{i+1})g_{\param}(x_{i+1},\bfy_{i+1})\right| \\
=  \exp\left((s+r) \phi(\param)+ \pscal{\psi(\param)}{\sum_{i=-r+1}^{s}
  S(x_i, x_{i+1}, \bfy_{i+1})}\right)\eqsp.
\end{multline*}
Since, by H\ref{BOEMsupp:assum:exp}, $\phi$
and $\psi$ are continuous, and since $\paramset$ is compact, there exist constants $C_{1}$ and $C_{2}$ such that,
\begin{multline*}
  \sup_{\param\in \K} \left|\prod_{i=-r+1}^{s}
    m_{\param}(x_i, x_{i+1})g_{\param}(x_{i+1},\bfy_{i+1})\right| \\
  \leq C_{1}\exp\left(C_{2}\sum_{i=-r+1}^{s} \sup_{x,x'}
    |S(x,x,',\bfy_{i+1})|\right)\eqsp.
\end{multline*}
  
Since the measure $\chi(\rmd x_{-r})\prod_{i=-r+1}^s\lambda(\rmd x_i)$ is finite, the dominated convergence
theorem now implies that $\ell_{\param,s}^{\chi,-r}(\bfy)$ is continuous on
$\paramset$.

For the proof of~(\ref{BOEMsupp:lem:cont-log-lik:moment}), let us apply the dominated
convergence theorem again.  Since $\paramset$ is compact, for any $\bfy \in
\Yset^\Zset$, $\param \mapsto \delta_{\param}(\bfy)$ is uniformly continuous
and $\underset{\eta \to
  0}{\lim}\underset{\left|\param-\param^{\prime}\right|<\eta}{\sup}\left|\delta_{\param}(\bfy)-\delta_{\param'}(\bfy)\right|
=0$. In addition, we have by Lemma \ref{BOEMsupp:lem:bound-log-lik}
\begin{multline*}
  \underset{\{\param,\param' \in \paramset;
    \left|\param-\param^{\prime}\right|<\eta\}}{\sup}\left|\delta_{\param}(\bfy)-\delta_{\param'}(\bfy)\right|
  \\
  \leq 2 \sup_{\param \in \paramset} |\delta_{\param}(\bfy)| \leq
  \frac{4}{(1-\rho)} + 2 \left\{\left|\log \sigma_+\, b_+(y_1)\right| +
    \left|\log \sigma_- \, b_-(y_1)\right|\right\} \eqsp.
\end{multline*}
Under H\ref{BOEMsupp:assum:obs}, this upper bound is $\PPim$-integrable.
This concludes the proof.
\end{proof}

\begin{theorem}\label{BOEMsupp:th:likelihood}
  Assume H\ref{BOEMsupp:assum:exp}-\ref{BOEMsupp:assum:strong} and H\ref{BOEMsupp:assum:obs}. 
   Define the function $\ell: \paramset \to \Rset$ by $\ell(\param) \eqdef
  \CE[]{}{}{\delta_\param(\bfY)}$, where $\delta_\param(\bfy)$ is defined
  in Lemma~\ref{BOEMsupp:lem:bound-log-lik}.
 \begin{enumerate}[(i)]
 \item \label{BOEMsupp:th:likelihood:continuous}The function $\param \mapsto
   \ell(\param)$ is continuous on $\paramset$.
 \item \label{BOEMsupp:th:likelihood:convergence} For any initial distribution $\chi$ on
   $(\Xset,\sigmaX)$
\begin{equation}\label{eq:log-lik}
\left|\frac{1}{T} \ell^{\chi,0}_{\param,T}(\bfY)-\ell(\param)\right| \underset{T\rightarrow +\infty}{\longrightarrow}0\eqsp,\quad\ps{\PPim}
\end{equation}
where $\ell^{\chi,0}_{\param,T}(\bfY)$ is defined in \eqref{BOEMsupp:eq:def:loglikelihood}.
\item \label{BOEMsupp:th:likelihood:uniformconvergence}For any initial distribution $\chi$ on $(\Xset,\sigmaX)$ 
\begin{equation}\label{BOEMsupp:eq:log-lik:unif}
\underset{\param\in\paramset}{\sup}\left|\frac{1}{T} \ell^{\chi,0}_{\param,T}(\bfY)-\ell(\param)\right| \underset{T\rightarrow +\infty}{\longrightarrow}0\eqsp,\quad\ps{\PPim}
\end{equation}
\end{enumerate}
\end{theorem}
\begin{proof}
  \textit{(i)} is proved in Lemma~\ref{BOEMsupp:lem:cont-log-lik}. 

\textit{(ii)} By \eqref{BOEMsupp:eq:loglikelihood-delta}, for any $T>0$, we have, for any
$\bfy\in\Yset^\Zset$:
 \begin{align*}
\frac{1}{T} \ell_{\param,T}^{\chi,0}(\bfy) &= \frac{1}{T} \sum_{s=0}^{T-1}\delta_{\param,s}^{\chi,0}(\bfy) \\
&=\frac{1}{T} \sum_{s=0}^{T-1}\left(\delta_{\param,s}^{\chi,0}(\bfy) - \delta_{\param}(\shift^s\circ\bfy)\right) +  \frac{1}{T} \sum_{s=0}^{T-1}\delta_{\param}(\shift^s\circ\bfy) \eqsp.
\end{align*}
By Lemma~\ref{BOEMsupp:lem:forget-log-lik}\eqref{BOEMsupp:lem:forget-log-lik:claim2}, for any
$0\leq s\leq T-1$, $\left|\delta_{\param,s}^{\chi,0}(\bfY) -
  \delta_{\param}(\shift^s\circ\bfY)\right|\leq 2\frac{\rho^{s}}{1-\rho}$.
Since $\rho \in (0,1)$,
\[\underset{T\to\infty}{\lim}\;\frac{1}{T}\sum_{s=0}^{T-1}\left(\delta_{\param,s}^{\chi,0}(\bfY)
  - \delta_{\param}(\shift^s\circ\bfY)\right)=0\quad \ps{\PPim}\eqsp.\] By
Lemma~\ref{BOEMsupp:lem:bound-log-lik}
\[
\Expparam{\delta_{\param}(\bfY)}{} \leq \frac{2}{(1-\rho)} +\Expparam{ \left|\log
  \sigma_+ b_+(\bfY_1)\right|+\left|\log
    \sigma_-b_-(\bfY_1)\right|}{}\eqsp,
\]
and the rhs is finite under assumption H\ref{BOEMsupp:assum:obs}\eqref{BOEMsupp:assum:obs:b+:b-}.
By H\ref{BOEMsupp:assum:obs}(\ref{BOEMsupp:assum:obs:mix}), the ergodic theorem, see \cite[Theorem 24.1, p.314]{billingsley:1987}, concludes the proof.

\textit{(iii)} Since $\paramset$ is compact, (\ref{BOEMsupp:eq:log-lik:unif}) holds if
for any $\varepsilon>0$, any $\param' \in \paramset$, there exists $\eta>0$
such that
\begin{equation}\label{BOEMsupp:eq:cond-uniform-convergence}
\lim_{T \to +\infty} \sup_{\{\param; |\param - \param'| < \eta\} \cap \paramset}
\left| T^{-1} \ell_{\param,T}^{\chi,0}(\bfY) - T^{-1} \ell_{\param^\prime,T}^{\chi,0}(\bfY)\right| \leq
\varepsilon\eqsp,\quad \ps{\PPim}
\end{equation}
Let $\varepsilon>0$ and $\param^\prime\in\paramset$. 
Choose $\eta>0$ such that
\begin{equation}
  \label{BOEMsupp:eq:proof:BOEMsupp:th:likelihood:tool1}
  \Expparam{\sup_{\{\param\in\paramset; |\param - \param'| < \eta\}} \left|
 \delta_{\param}(\bfY) - \delta_{\param'}(\bfY)\right|}{}\leq \varepsilon\eqsp;
\end{equation}
such an $\eta$ exists by Lemma~\ref{BOEMsupp:lem:cont-log-lik}. By \eqref{BOEMsupp:eq:loglikelihood-delta}, we have, for any $\param\in\paramset$ such that $ |\param - \param'| < \eta$
\begin{equation}
\label{eq:diff-log-decomp}
  \left| \frac{1}{T} \ell_{\param,T}^{\chi,0}(\bfY) -
   \frac{1}{T} \ell_{\param^\prime,T}^{\chi,0}(\bfY) \right| \leq \frac{1}{T} \sum_{s=0}^{T-1}
 \left|    \delta_{\param,s}^{\chi,0}(\bfY) - 
    \delta_{\param',s}^{\chi,0}(\bfY)   \right|\eqsp.
\end{equation}
In addition, by Lemma~\ref{BOEMsupp:lem:forget-log-lik}\eqref{BOEMsupp:lem:forget-log-lik:claim2}
\begin{multline*}
  \sum_{s=0}^{T-1} \left|\delta_{\param,s}^{\chi,0}(\bfY) -
    \delta_{\param',s}^{\chi,0}(\bfY) \right| \\
  \leq 2 \sum_{s=0}^{T-1} \sup_{\param \in \paramset} \left|
    \delta_{\param,s}^{\chi,0}(\bfY) - \delta_{\param}(\shift^s\circ\bfY)
  \right| + \sum_{s=0}^{T-1}\left| \delta_{\param}(\shift^s\circ\bfY) - \delta_{\param'}(\shift^s\circ\bfY) \right|  \\
  \leq \frac{4}{(1-\rho)^2} +\sum_{s=0}^{T-1} \Xi(\shift^s \circ \bfY)
\end{multline*}
where $\Xi(\bfy) \eqdef \sup_{\{\param\in\paramset; |\param - \param'| <
  \eta\}} \left| \delta_{\param}(\bfy - \delta_{\param'}(\bfy) \right|$. This
implies that
\[
\lim_{T \to +\infty} \sup_{\{\param\in\paramset; |\param - \param'| <
  \eta\}}\frac{1}{T}\sum_{s=0}^{T-1} \left|\delta_{\param,s}^{\chi,0}(\bfY)
  -\delta_{\param',s}^{\chi,0}(\bfY) \right| \leq\lim_{T \to
  +\infty}\frac{1}{T}\sum_{s=0}^{T-1} \Xi(\shift^s \circ \bfY)\eqsp.
\]
Under H\ref{BOEMsupp:assum:obs}, the ergodic
theorem implies that the rhs converges $\ps{\PPim}$ to $\CE[]{}{}{\Xi(\bfY)}$, see
\cite[p.314]{billingsley:1987}. Then, using again
\eqref{BOEMsupp:eq:proof:BOEMsupp:th:likelihood:tool1},
 \[
 \lim_{T \to +\infty} \sup_{\{\param\in\paramset; |\param - \param'| < \eta\}}\frac{1}{T}\sum_{s=0}^{T-1} \left|\delta_{\param,s}^{\chi,0}(\bfY) -\delta_{\param',s}^{\chi,0}(\bfY) \right|\leq \varepsilon\eqsp,\quad \ps{\PPim}
 \]
 Then, \eqref{BOEMsupp:eq:cond-uniform-convergence} holds and this concludes the proof.
\end{proof}

\subsection{Limit of the normalized score}
\label{BOEMsupp:sec:normalized:score}
This section is devoted to the proof of the $\ps{\PPim}$ convergence of the normalized score $T^{-1} \nabla_\param\ell_{\param,T}^{\chi,0}(\bfY)$ to $\nabla_{\param}\ell(\param)$. This result is established under additional assumptions on the model.
\begin{hypS}\label{BOEMsupp:assum:regularity-grad}
\begin{enumerate}[(a)]
\item \label{assum:regularity:diff} For any $y\in\Yset$ and for all $(x,x^{\prime})\in\Xset^{2}$,
  $\param\mapsto g_{\param}(x,y)$  and $\param\mapsto m_{\param}(x,x^{\prime})$ are continuously differentiable on
  $\paramset$.
\item  \label{BOEMsupp:assum:regularity:phi}We assume that $\mathbb{E}\left[\phi(\bfY_{0})\right]<+\infty$ where
\begin{equation}\label{BOEMsupp:eq:def-phiu}
\phi(y) \eqdef \underset{\param\in\paramset}{\sup}\;\underset{(x,x^{\prime})\in\Xset^{2}}{\sup}\left|\nabla_{\param}\log m_{\param}(x,x^{\prime})+ \nabla_{\param}\log g_{\param}(x^{\prime},y)\right|\eqsp.
\end{equation} 
\end{enumerate}
\end{hypS}

\begin{lemma}\label{BOEMsupp:lem:score:differentiabilite}
  Assume S\ref{BOEMsupp:assum:regularity-grad}. For any initial distribution $\chi$, any
  integers $s,r \geq 0$ and any $\bfy \in \Yset^\Zset$ such that
  $\phi(\bfy_u)<+\infty$ for any $u \in \Zset$, the function $\param \mapsto
  \ell_{\param,s}^{\chi, s-r}(\bfy)$ is continuously differentiable on
  $\paramset$ and
  \begin{equation*}
    \nabla_\param \ell_{\param,s}^{\chi, s-r}(\bfy) 
    = \sum_{u=s-r}^s\smoothfunc{\chi,s-r-1}{\param,u,s}(\Upsilon_{\param},\bfy)
    \eqsp,
  \end{equation*}
  where $\Upsilon_\param$ is the function defined on $\Xset^{2}\times\Yset$ by  
\[
\Upsilon_{\param}: (x,x^{\prime},y)\mapsto \nabla_\param \log
\left\{m_{\param}(x,x^{\prime})g_\param(x^{\prime},y)\right\} \eqsp.
\]
\end{lemma}
\begin{proof}
  Under S\ref{BOEMsupp:assum:regularity-grad}, the dominated convergence theorem implies
  that the function $\param \mapsto \ell_{\param,s}^{\chi, s-r}(\bfy)$ is
  continuously differentiable and its derivative is obtained by permutation of
  the gradient and integral operators.
\end{proof}
 
\begin{lemma}\label{BOEMsupp:lem:score:LimiteDerivee}
  Assume H\ref{BOEMsupp:assum:strong}  and
  S\ref{BOEMsupp:assum:regularity-grad}.
  \begin{enumerate}[(i)]
  \item \label{BOEMsupp:BOEMsupp:lem:score:LimiteDerivee:claim1} There exists a function $\xi :
    \Yset^{\mathbb{Z}} \to \mathbb{R}_+$ such that for any $s\geq 0$ and any
    $r, r^{\prime}\geq s$, any initial distribution $\chi,\chi'$ on $\Xset$ and
    any $\bfy \in \Yset^\Zset$ such that $\phi(\bfy_u)<+\infty$ for any $u \in
    \Zset$,
\[
\sup_{\param \in \paramset} \left| \nabla_\param
  \delta_{\param,s}^{\chi,s-r}(\bfy) - \nabla_\param
  \delta_{\param,s}^{\chi',s-r'}(\bfy) \right| \leq \frac{16 \rho^{-1/4}}{1-\rho}
\rho^{(r' \wedge r)/4} \ \xi(\bfy) \eqsp,
\]
where
\begin{equation}\label{BOEMsupp:eq:def-xi}
\xi(\bfy) \eqdef \sum_{u \in \Zset} \phi(\bfy_u)\rho^{|u|/4}  \eqsp.
\end{equation}
\item \label{BOEMsupp:lem:score:LimiteDerivee:claim2} For any $\bfy\in \Yset^\Zset$
  satisfying $\xi(\bfy)<+\infty$, the function $\param \mapsto
  \delta_\param(\bfy)$ given by
  Lemma~\ref{BOEMsupp:lem:forget-log-lik}\eqref{BOEMsupp:lem:forget-log-lik:claim2} is
  continuously differentiable on $\paramset$; and, for any $\param \in \paramset$,
  any initial distribution $\chi$ and any integers $r \geq s \geq 0$,
\[
\sup_{\param \in \paramset} \left| \nabla_\param
  \delta_{\param,s}^{\chi,s-r}(\bfy) - \nabla_\param\delta_\param(\bfy \circ
  \shift^s)\right| \leq \frac{16 \rho^{-1/4}}{1-\rho} \rho^{r/4} \ \xi(\bfy)
\eqsp.
\]
  \end{enumerate}
\end{lemma}
\begin{proof} \eqref{BOEMsupp:BOEMsupp:lem:score:LimiteDerivee:claim1}
  By definition of $\delta_{\param,s}^{\chi, s-r}(\bfy)$, see
  \eqref{BOEMsupp:eq:delta-loglikelihood} and Lemma~\ref{BOEMsupp:lem:score:differentiabilite},
  \begin{align*}
    \nabla_\param \delta_{\param,s}^{\chi,s-r}(\bfy) & - \nabla_\param
    \delta_{\param,s}^{\chi',s-r'}(\bfy) \\
    & = \nabla_{\param}\ell^{\chi,s-r}_{\param,s+1}(\bfy) -
    \nabla_{\param}\ell^{\chi,s-r}_{\param,s}(\bfy) -
    \nabla_{\param}\ell^{\chi',s-r'}_{\param,s+1}(\bfy) +
    \nabla_{\param}\ell^{\chi',s-r'}_{\param,s}(\bfy) \\
    & = \sum_{u=s-r}^s \left(
      \smoothfunc{\chi,s-r-1}{\param,u,s+1}(\Upsilon_{\param},\bfy) -
      \smoothfunc{\chi,s-r-1}{\param,u,s}(\Upsilon_{\param},\bfy) \right) \\
   & - \sum_{u=s-r'}^s \left(
      \smoothfunc{\chi^{\prime},s-r^{\prime}-1}{\param,u,s+1}(\Upsilon_{\param},\bfy)
      -
      \smoothfunc{\chi^{\prime},s-r^{\prime}-1}{\param,u,s}(\Upsilon_{\param},\bfy) \right) \\
    &+ \smoothfunc{\chi,s-r-1}{\param,s+1,s+1}(\Upsilon_{\param},\bfy) -
    \smoothfunc{\chi^{\prime},s-r^{\prime}-1}{\param,s+1,s+1}(\Upsilon_{\param},\bfy)
    \eqsp.
  \end{align*}

  We can assume without loss of generality that $r' \leq r$ so that
\begin{multline*}
  \nabla_\param \delta_{\param,s}^{\chi,s-r}(\bfy) - \nabla_\param
  \delta_{\param,s}^{\chi',s-r'}(\bfy)  \\
  = \sum_{u=s-r}^{s-r^{\prime}-1} \left\{ \smoothfunc{\chi,s-r-1}{\param,u,s+1}(\Upsilon_{\param},\bfy) -
      \smoothfunc{\chi,s-r-1}{\param,u,s}(\Upsilon_{\param},\bfy) \right\} +
  \smoothfunc{\chi,s-r-1}{\param,s+1,s+1}(\Upsilon_{\param},\bfy) -
    \smoothfunc{\chi^{\prime},s-r^{\prime}-1}{\param,s+1,s+1}(\Upsilon_{\param},\bfy)\\+\sum_{u=s-r^{\prime}}^{s}\left\{\smoothfunc{\chi,s-r-1}{\param,u,s+1}(\Upsilon_{\param},\bfy) -
      \smoothfunc{\chi,s-r-1}{\param,u,s}(\Upsilon_{\param},\bfy)-
   \smoothfunc{\chi^{\prime},s-r^{\prime}-1}{\param,u,s+1}(\Upsilon_{\param},\bfy) +
     \smoothfunc{\chi^{\prime},s-r^{\prime}-1}{\param,u,s}(\Upsilon_{\param},\bfy)\right\}\eqsp.
\end{multline*}
Under H\ref{BOEMsupp:assum:strong} and S\ref{BOEMsupp:assum:regularity-grad}, 
Remark~\ref{BOEMsupp:rem:lgn} can be applied and for any $s-r\leq u \leq
s-r^{\prime}-1$,
    \begin{equation*}
    \left|\smoothfunc{\chi,s-r-1}{\param,u,s+1}(\Upsilon_{\param},\bfy) -
      \smoothfunc{\chi,s-r-1}{\param,u,s}(\Upsilon_{\param},\bfy)\right|\leq 2\rho^{s-u}\phi(\bfy_u),
    \end{equation*}
    where $\phi_u(\bfy)$ is defined in \eqref{BOEMsupp:eq:def-phiu}. Similarly, by
   Remark~\ref{BOEMsupp:rem:lgn}
   \begin{equation*}
\left| \smoothfunc{\chi,s-r-1}{\param,s+1,s+1}(\Upsilon_{\param},\bfy) -
    \smoothfunc{\chi^{\prime},s-r^{\prime}-1}{\param,s+1,s+1}(\Upsilon_{\param},\bfy)\right|\leq 2\rho^{r^{\prime}+1}\phi(\bfy_{s+1})\eqsp.
\end{equation*}
For any $s-r^{\prime}\leq u \leq s$, by Remark~\ref{BOEMsupp:rem:lgn},
\begin{multline*}
  \left|\smoothfunc{\chi,s-r-1}{\param,u,s+1}(\Upsilon_{\param},\bfy) -
      \smoothfunc{\chi^{\prime},s-r^{\prime}-1}{\param,u,s+1}(\Upsilon_{\param},\bfy)+ \smoothfunc{\chi^{\prime},s-r^{\prime}-1}{\param,u,s}(\Upsilon_{\param},\bfy)-\smoothfunc{\chi,s-r-1}{\param,u,s}(\Upsilon_{\param},\bfy)\right|\\
      \leq\left|\smoothfunc{\chi,s-r-1}{\param,u,s+1}(\Upsilon_{\param},\bfy) -
      \smoothfunc{\chi^{\prime},s-r^{\prime}-1}{\param,u,s+1}(\Upsilon_{\param},\bfy)\right|+\left|\smoothfunc{\chi^{\prime},s-r^{\prime}-1}{\param,u,s}(\Upsilon_{\param},\bfy)-\smoothfunc{\chi,s-r-1}{\param,u,s}(\Upsilon_{\param},\bfy)\right|\\
  \leq 4\rho^{u+r^{\prime}-s}\phi(\bfy_u)
\end{multline*}
and by Remark~\ref{BOEMsupp:rem:lgn},
\begin{multline*}
\left|\smoothfunc{\chi,s-r-1}{\param,u,s+1}(\Upsilon_{\param},\bfy) -
      \smoothfunc{\chi^{\prime},s-r^{\prime}-1}{\param,u,s+1}(\Upsilon_{\param},\bfy)+ \smoothfunc{\chi^{\prime},s-r^{\prime}-1}{\param,u,s}(\Upsilon_{\param},\bfy)-\smoothfunc{\chi,s-r-1}{\param,u,s}(\Upsilon_{\param},\bfy)\right|\\
      \leq\left|\smoothfunc{\chi,s-r-1}{\param,u,s+1}(\Upsilon_{\param},\bfy) -\smoothfunc{\chi,s-r-1}{\param,u,s}(\Upsilon_{\param},\bfy)\right|+\left| \smoothfunc{\chi^{\prime},s-r^{\prime}-1}{\param,u,s+1}(\Upsilon_{\param},\bfy)-\smoothfunc{\chi^{\prime},s-r^{\prime}-1}{\param,u,s}(\Upsilon_{\param},\bfy)\right|\\
\leq 4\rho^{s-u}\phi(\bfy_u)\eqsp.
\end{multline*}
Hence,
\begin{equation*}
\left| \nabla_\param \delta_{\param,s}^{\chi,s-r}(\bfy) - \nabla_\param
  \delta_{\param,s}^{\chi',s-r'}(\bfy)   \right| \leq 2\sum_{u=s-r}^{s-r^{\prime}-1}\rho^{s-u}\phi(\bfy_u) + 4\sum_{u=s-r^{\prime}}^{s+1}\left(\rho^{u+r^{\prime}-s}\wedge\rho^{s-u}\right)\phi(\bfy_u)\eqsp.
\end{equation*}

Furthermore,
\begin{align*}
   \sum_{u=s-r^{\prime}}^{s+1}\phi(\bfy_u)& \left(\rho^{u+r^{\prime}-s}\wedge
    \rho^{s-u}\right) \\
& \leq \sum_{ s-r' \leq u\leq \lfloor s-r^{\prime}/2
    \rfloor}\rho^{s-u} \phi(\bfy_u)
  +\sum_{u\geq \lfloor s-r^{\prime}/2\rfloor}\rho^{u+r^{\prime}-s}\phi(\bfy_u) \\
  &  \leq \rho^{r'/2} \sum_{u \in \Zset} \phi(\bfy_u) \rho^{|u|/4} \  \cdots \\
& \times \left(
    \sum_{ u\leq \lfloor s-r^{\prime}/2 \rfloor} \rho^{s-u -r'/2 -|u|/4} +
    \sum_{
      \lfloor s-r^{\prime}/2 \rfloor + 1 \leq u \leq s+1} \rho^{u+r'/2-s -|u|/4}  \right)      \\
  &  \leq 2 \frac{\rho^{(r'-1)/4}}{1-\rho} \ \sum_{u \in \Zset} \phi(\bfy_u)
  \rho^{|u|/4}\eqsp,
 \end{align*}
 where we used that $\sup_{s-r' \leq u \leq \lfloor s-r'/2 \rfloor } |u| \leq
 r'$ and $\sup_{ \lfloor s-r^{\prime}/2 \rfloor + 1 \leq u \leq s+1} |u| \leq
 r'+1$.  Moreover, upon noting that $-u/2+(s+1)/2 \leq s-u-r^{\prime}/2$ when
 $u \leq s-r^{\prime}-1$,
\begin{align*}
  \sum_{u=s-r}^{s-r^{\prime}-1}\phi(\bfy_u)\rho^{s-u}&\leq
  \rho^{r^{\prime}/2}\sum_{u=s-r}^{s-r^{\prime}-1}\phi(\bfy_u)\rho^{s-u-r^{\prime}/2}
   \\
& \leq  \rho^{r^{\prime}/2}\sum_{u=s-r}^{s-r^{\prime}-1}\phi(\bfy_u)\rho^{-u/2+(s+1)/2}\\
  &\leq  \rho^{r^{\prime}/2}\rho^{(s+1)/2}\sum_{u=s-r}^{s-r^{\prime}-1}\phi(\bfy_u)\rho^{|u|/2}\eqsp,
\end{align*}
where we used that $s-r'-1 \leq 0$ in the last inequality.

Hence,
\begin{equation}\label{eq:uniformdiffelta}
\sup_{\param \in \paramset} \left|  \nabla_\param \delta_{\param,s}^{\chi,s-r}(\bfy) - \nabla_\param
  \delta_{\param,s}^{\chi',s-r'}(\bfy)    \right| \leq \frac{16}{1-\rho} \rho^{(r'-1)/4} \ \sum_{u \in \Zset} \phi(\bfy_u)
  \rho^{|u|/4} \eqsp.
\end{equation}
\eqref{BOEMsupp:lem:score:LimiteDerivee:claim2} Let $\bfy \in \Yset^\Zset$ such that
$\xi(\bfy) < +\infty$. Then for any $u \in \Zset$, $\phi(\bfy_u) < +\infty$. By
Lemma~\ref{BOEMsupp:lem:score:differentiabilite} and Eq.~(\ref{BOEMsupp:eq:delta-loglikelihood}),
the functions $\{\param \mapsto \delta_{\param,0}^{\chi,-r}(\bfy)\}_{r \geq 0}$
are $C^1$ functions on $\paramset$. By \eqref{BOEMsupp:BOEMsupp:lem:score:LimiteDerivee:claim1},
there exists a function $\param \mapsto \tilde \delta_\param(\bfy)$ such that 
\[
\lim_{r \to +\infty} \sup_{\param \in \paramset}
\left|\nabla_\param\delta_{\param,0}^{\chi,-r}(\bfy) - \tilde
  \delta_\param(\bfy)\right| =0 \eqsp.
\]
Furthermore, by Lemma~\ref{BOEMsupp:lem:forget-log-lik}, 
\[
\lim_{r \to +\infty} \sup_{\param \in \paramset}
\left|\delta_{\param,0}^{\chi,-r}(\bfy) -
  \delta_\param(\bfy)\right| =0 \eqsp.
\]
Then, $\param \mapsto \delta_\param(\bfy)$ is $C^1$ on $\paramset$ and for any
$\param \in \paramset$, $\tilde \delta_\param(\bfy) = \nabla_\param
\delta_\param(\bfy)$.

We thus proved that for any $\bfy \in \Yset^\Zset$ such that $\xi(\bfy)<
+\infty$ and for any initial distribution $\chi$,
\begin{equation}
  \label{BOEMsupp:eq:lem:score:Limit:claim2:tool1}
  \lim_{r \to +\infty} \sup_{\param \in \paramset}
\left|\nabla_\param\delta_{\param,0}^{\chi,-r}(\bfy) - \nabla_\param
  \delta_\param(\bfy)\right| =0 \eqsp.
\end{equation}
Observe that by definition,
$\nabla_\param\delta_{\param,s}^{\chi,s-r}(\bfy) =
\nabla_\param\delta_{\param,0}^{\chi,-r}(\shift^s\circ\bfy)$. This property,
combined with
Lemma~\ref{BOEMsupp:lem:score:LimiteDerivee}\eqref{BOEMsupp:BOEMsupp:lem:score:LimiteDerivee:claim1},
yields
\begin{multline*}
  \sup_{\param \in \paramset} \left|
    \nabla_\param\delta_{\param,s}^{\chi,s-r}(\bfy)
    - \nabla_\param\delta_\param(\shift^s\circ\bfy) \right|  \\
  \leq \frac{16 \rho^{-1/4}}{1-\rho} \rho^{r/4} \ \xi( \bfy) + \sup_{\param \in
    \paramset}
  \left|\nabla_\param\delta_{\param,0}^{\chi,-r'}(\shift^s\circ\bfy) -
    \nabla_\param\delta_\param(\shift^s\circ\bfy) \right| \eqsp.
\end{multline*}
Since $\xi(\shift^s\circ\bfy)< +\infty$, when $r' \to +\infty$, the second term tends to zero by
(\ref{BOEMsupp:eq:lem:score:Limit:claim2:tool1}) - for fixed $\bfy, s$ and $\chi$ -.
This concludes the proof.
\end{proof}

\begin{lemma}\label{BOEMsupp:lem:bound-score}
  \begin{enumerate}[(i)]
  \item \label{BOEMsupp:lem:bound-score:claim1} Assume S\ref{BOEMsupp:assum:regularity-grad}. For
    any $\bfy \in \Yset^\Zset$ such that $\phi(\bfy_u) < +\infty$ for any $u \in
    \Zset$, for any integers $r,s\geq 0$,
\[
\sup_{\param \in \paramset}
\left|\nabla_\param\delta_{\param,s}^{\chi,s-r}(\bfy) \right| \leq 2\sum_{u=s-r}^{s+1}\phi(\bfy_{u})
\eqsp.
\]
\item \label{BOEMsupp:lem:bound-score:claim2} Assume H\ref{BOEMsupp:assum:strong} and
  S\ref{BOEMsupp:assum:regularity-grad}. Then, for any $y\in\Yset^\Zset$ such that
  $\xi(\bfy)<+\infty$ and for any $r\geq 0$,
\[
\sup_{\param \in \paramset} |\nabla_\param\delta_{\param}(\bfy)|\leq
2\sum_{u=-r}^{1}\phi(\bfy_{u}) + \frac{16 \rho^{-1/4}}{1-\rho} \xi(\bfy) \rho^{r/4} \eqsp,\]
  \end{enumerate}
  where $\xi(\bfy)$ is defined in Lemma~\ref{BOEMsupp:lem:score:LimiteDerivee}.
\end{lemma}
\begin{proof}
  \eqref{BOEMsupp:lem:bound-score:claim1} By (\ref{BOEMsupp:eq:delta-loglikelihood}) and
  Lemma~\ref{BOEMsupp:lem:score:differentiabilite},
\begin{multline*}
  \left| \nabla_\param\delta^{\chi,s-r}_{\param,s}(\bfy) \right| =   \left| \nabla_{\param}\ell^{\chi,s-r}_{\param,s+1}(\bfy) -  \nabla_{\param}\ell^{\chi,s-r}_{\param,s}(\bfy) \right| \\
  \leq 2 \sum_{u=s-r}^{s+1} \left| \frac{\int \chi(\rmd
      x_{s-r})L_{\param,s-r:u-1}(x_{s-r},\rmd
      x_{u})\nabla_\param \log \left[m_{\param}(x_{u-1},x_{u})g_\param(x_{u},\bfy_{u})\right]
      L_{\param,u:s-1}(x_{u},\Xset)}{\int \chi(\rmd
      x_{s-r})L_{\param,s-r:s-1}(x_{s-r},\Xset)} \right| \eqsp.
\end{multline*}
The proof is concluded upon noting that for any $s-r\leq u \leq s+1$,
\[\left|
  \frac{\int \chi(\rmd
    x_{s-r})g_{\param}(x_{s-r},y_{s-r})L_{\param,s-r:u-1}(x_{s-r},\rmd
    x_{u})\nabla_\param \log g_\param(x_{u},\bfy_{u})
    L_{\param,u:s-1}(x_{u},\Xset)}{\int \chi(\rmd
    x_{s-r})g_{\param}(x_{s-r},\bfy_{s-r})L_{\param,s-r:s-1}(x_{s-r},\Xset)}
\right|\] is upper bounded by $\phi(\bfy_{u})$.

\eqref{BOEMsupp:lem:bound-score:claim2} is a consequence of
Lemma~\ref{BOEMsupp:lem:score:LimiteDerivee}\eqref{BOEMsupp:lem:score:LimiteDerivee:claim2} and
Lemma~\ref{BOEMsupp:lem:bound-score}\eqref{BOEMsupp:lem:bound-score:claim1}.
\end{proof}

\begin{theorem}\label{BOEMsupp:th:score}
  Assume H\ref{BOEMsupp:assum:strong}, H\ref{BOEMsupp:assum:obs}(\ref{BOEMsupp:assum:obs:mix}) and S\ref{BOEMsupp:assum:regularity-grad}.
 \begin{enumerate}[(i)]
 \item \label{BOEMsupp:th:score:C1} For any $ T \geq 0$ and any distribution $\chi$ on $\Xset$, the
   functions $\param \mapsto \ell^{\chi,0}_{\param,T}(\bfY)$ and $\param
   \mapsto \ell(\param)$ are continuously differentiable
   $\ps{\PPim}$
     \item \label{BOEMsupp:th:score:convergence}For any initial
distribution $\chi$ on $(\Xset,\sigmaX)$, 
\begin{equation}\label{eq:score}
\frac{1}{T}\nabla_{\param}  \ell^{\chi,0}_{\param,T}(\bfY) \underset{T\rightarrow +\infty}{\longrightarrow}\nabla_{\param}\ell(\param)\quad\ps{\PPim}
\end{equation}
\end{enumerate}
\end{theorem}
\begin{proof}
  By (\ref{BOEMsupp:eq:loglikelihood-delta}) and
  Lemma~\ref{BOEMsupp:lem:score:differentiabilite}, for any $\bfy$ such that $\phi(\bfy_u)
  < +\infty$ for any $u \in \Zset$, $\ell_{\param,T}^{\chi,0}(\bfy)$ and
  $\delta_{\param,s}^{\chi,0}(\bfy)$ are continuously differentiable and
  \eqref{BOEMsupp:eq:loglikelihood-delta} implies
\begin{equation*}
  \nabla_\param \ell_{\param,T}^{\chi,0}(\bfy) = \sum_{s=0}^{T-1} \nabla_\param\delta_{\param,s}^{\chi,0}(\bfy)\eqsp.
\end{equation*}
This decomposition leads to
\begin{equation}\label{BOEMsupp:eq:decomp-grad-log}
\frac{1}{T} \nabla_\param\ell_{\param,T}^{\chi,0}(\bfY) = \frac{1}{T} \sum_{s=0}^{T-1}
\left(\nabla_\param\delta_{\param,s}^{\chi,0}(\bfY) -\nabla_\param\delta_{\param}(\shift^s\circ\bfY)\right) + \frac{1}{T}\sum_{s=0}^{T-1}\nabla_\param\delta_{\param}(\shift^s\circ\bfY)\eqsp.
\end{equation}

Consider the first term of the rhs of \eqref{BOEMsupp:eq:decomp-grad-log}. Since $\bfY$
is a stationary process, assumption S\ref{BOEMsupp:assum:regularity-grad}\eqref{BOEMsupp:assum:regularity:phi} implies that
$\Expparam{\xi(\bfY)}{}<+\infty$, where $\xi$ is defined by
\eqref{BOEMsupp:eq:def-xi}. Then, $\xi(\bfY)<+\infty\quad \ps{\PPim}$ and by
Lemma~\ref{BOEMsupp:lem:score:LimiteDerivee}\eqref{BOEMsupp:lem:score:LimiteDerivee:claim2}, for
any $0\leq s\leq T-1$,
 \[\left|\nabla_\param\delta_{\param,s}^{\chi,0}(\bfY) -\nabla_\param\delta_{\param}(\shift^s\circ\bfY)\right|\leq \xi(\bfY) \frac{16 \rho^{-1/4}}{1-\rho} \rho^{s/4}\eqsp.\] 
 Therefore
 \[\frac{1}{T}\sum_{s=0}^{T-1}\left|\nabla_\param\delta_{\param,s}^{\chi,0}(\bfY) -\nabla_\param\delta_{\param}(\shift^s\circ\bfY)\right|\leq \frac{1}{T}\xi(\bfY)  \frac{16 \rho^{-1/4}}{1-\rho} \frac{1}{1-\rho^{1/4}}   \eqsp,\]
 and
 \[
 \underset{T\to\infty}{\lim}\;\frac{1}{T}\sum_{s=0}^{T-1}\left(\nabla_\param\delta_{\param,s}^{\chi,0}(\bfY) -\nabla_\param\delta_{\param}(\shift^s\circ\bfY)\right)=0\eqsp,\quad \ps{\PPim}
 \]
Finally, consider the second term of the rhs of \eqref{BOEMsupp:eq:decomp-grad-log}. By Lemma~\ref{BOEMsupp:lem:bound-score} (applied with $r=1$), $\Expparam{\left|\nabla_\param\delta_{\param}(\bfY)\right|}{}<+\infty$. Under  H\ref{BOEMsupp:assum:obs}, the ergodic theorem (see \cite[Theorem 24.1, p.314]{billingsley:1987}) states that 
\[
\underset{T\to\infty}{\lim}\;\frac{1}{T}\sum_{s=0}^{T-1}\nabla_\param\delta_{\param}(\shift^s\circ\bfY)= \Expparam{\nabla_\param\delta_{\param}(\bfY)}{}\eqsp,\quad \ps{\PPim}
\]
Then, by \eqref{BOEMsupp:eq:decomp-grad-log} and the above discussion,
\[
\underset{T\to\infty}{\lim}\;\frac{1}{T}\nabla_\param\ell_{\param,T}^{\chi,0}(\bfY) = \Expparam{\nabla_\param\delta_{\param}(\bfY)}{}\eqsp,\quad\ps{\PPim}
\]
By Lemma~\ref{BOEMsupp:lem:bound-score}, applied with $r=0$,
\[\underset{\param\in\paramset}{\sup}\left|\nabla_\param\delta_{\param}(\bfY)\right|\leq  2\left[\phi(Y_{0})+\phi(Y_{1})\right] + \xi(\bfY)\rho^{1/2}\eqsp,\]
and the rhs is integrable under the stated assumptions. Therefore, by the
dominated convergence theorem,
$\Expparam{\nabla_\param\delta_{\param}(\bfY)}{} =
\nabla_\param\Expparam{\delta_{\param}(\bfY)}{} = \nabla_\param
\ell(\param)\eqsp.$ This concludes the proof.
\end{proof}

\newpage

\section{Additional experiments}\label{BOEMsupp:sec:appli}
In this section, we provide additional plots for the applications studied in \cite[Section~$3$]{lecorff:fort:2011}.
\subsection{Linear Gaussian model}
\label{BOEMsupp:sec:finiteLGM}
Figure~\ref{BOEMsupp:fig:init} illustrates the fact that the convergence properties of the BOEM do not depend on the initial distribution $\chi$ used in each block. Data are sampled using $\phi = 0.97$, $\sigma_{u}^{2} = 0.6$ and
$\sigma_{v}^{2} = 1$. All runs are started with $\phi = 0.1$, $\sigma_{u}^{2} =
1$ and $\sigma_{v}^{2} = 2$. Figure~\ref{BOEMsupp:fig:init} displays the estimation of $\phi$ by the
averaged BOEM algorithm with $\tau_{n} \sim n$ and $\tau_{n} \sim n^{1.5}$, over $100$
independent Monte Carlo runs as a function of the number of blocks.  We
consider first the case when $\chi$ is the stationary distribution of the
hidden process i.e. $\chi \equiv \mathcal{N}(0, (1-\phi^2)^{-1} \sigma_u^2)$, and the case when $\chi$ is the filtering distribution obtained at the
end of the previous block, computed with the Kalman filter. The estimation error is similar for both initialization schemes, even when $\phi$ is close to $1$ and for any choice of $\{\tau_{n}\}_{n\geq 1}$.
\begin{figure}[!h]
   \centering
   \subfloat[$\tau_{n}\sim n$]{\includegraphics[width=0.8\textwidth]{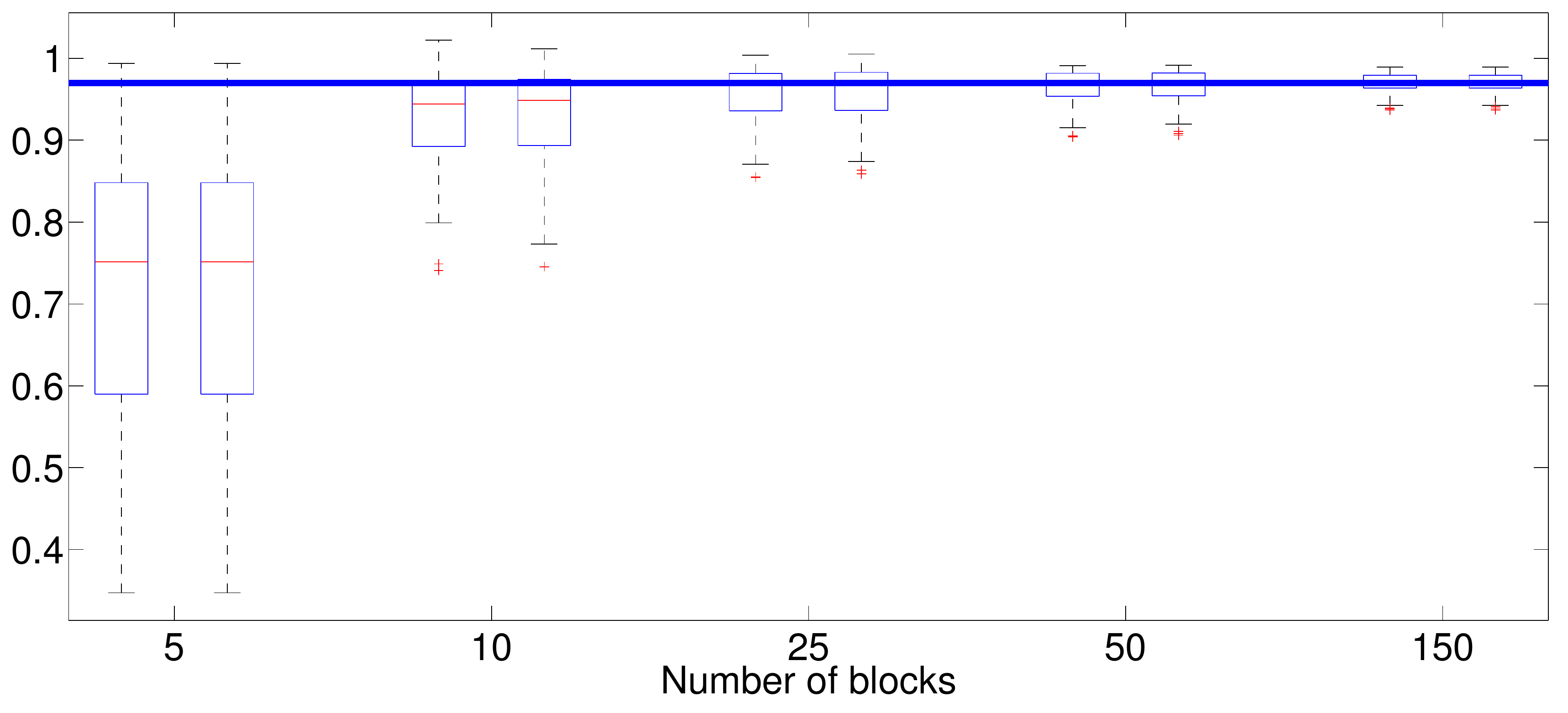}}\\
   \subfloat[$\tau_{n}\sim n^{1.5}$]{\includegraphics[width=0.8\textwidth]{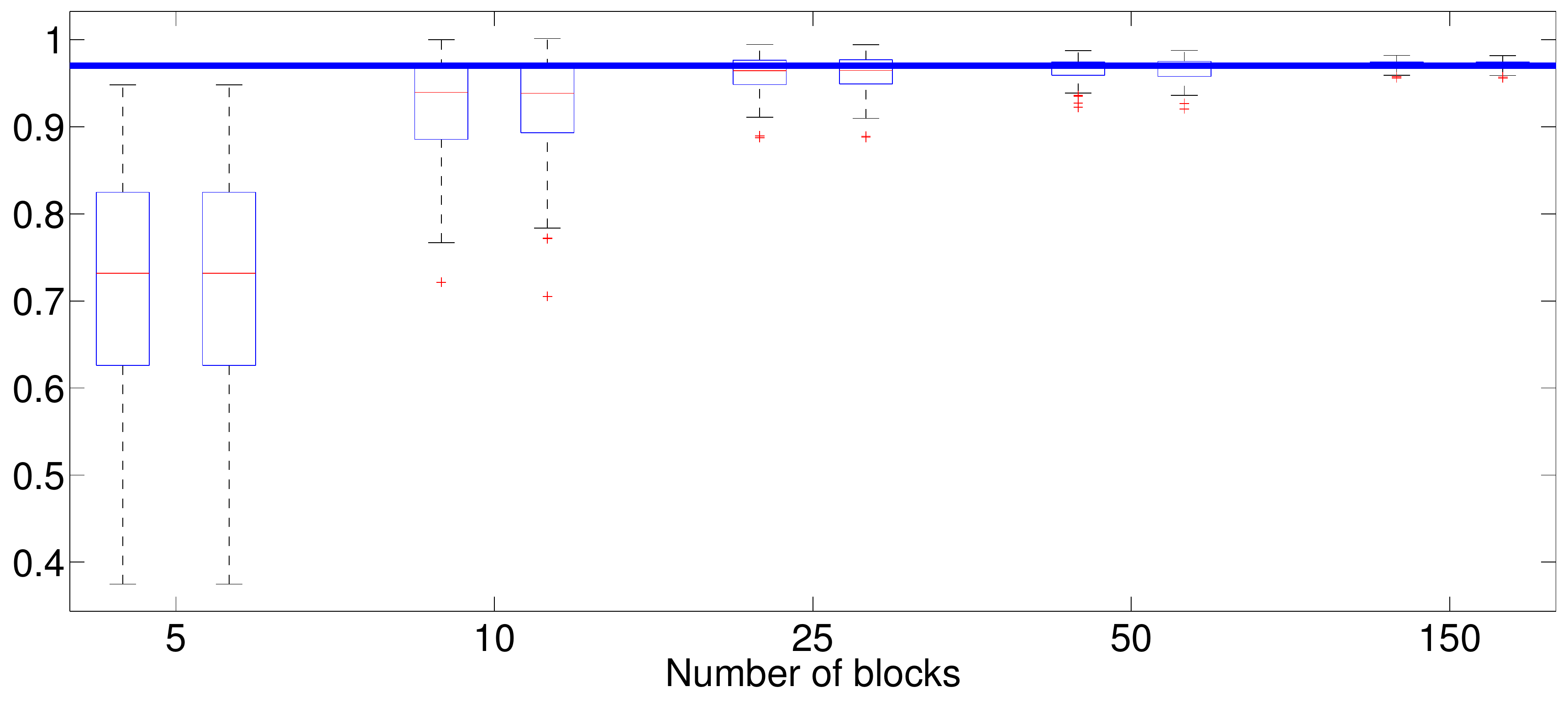}}
   \caption{Estimation of  $\phi$  after $5, 10, 25, 50$ and $150$ blocks, with two different initialization schemes: the stationary distribution (left) and the filtering distribution at the end of the previous block (right). The boxplots are computed with $100$ Monte Carlo runs.}
   \label{BOEMsupp:fig:init}
 \end{figure}

The theoretical analysis of BOEM says that a sufficient condition for
  convergence is the increasing size of the blocks.  On Figure~\ref{BOEMsupp:fig:blocksizes}, we compare different strategies for the definition of
  $\tau_n \eqdef T_{n} - T_{n-1}$. A slowly increasing sequence $\{\tau_{n}\}_{n\ge 0}$ is compared to different strategies using the same number of observations within each block. We consider the Linear Gaussian model:
\begin{equation*}
  X_{t+1} = \phi X_t + \sigma_uU_t\eqsp, \qquad \qquad Y_t = X_t + \sigma_vV_t\eqsp,
\end{equation*}
where $X_0\sim\mathcal{N}\left(0,\sigma_u^2 (1-\phi^2)^{-1}\right)$,
$\{U_t\}_{t\geq 0},\{V_t\}_{t\geq 0}$ are i.i.d. standard Gaussian r.v., independent from
$X_0$. Data are sampled using $\phi = 0.9$, $\sigma_{u}^{2} = 0.6$ and
$\sigma_{v}^{2} = 1$. All runs are started with $\phi = 0.1$, $\sigma_{u}^{2} =
1$ and $\sigma_{v}^{2} = 2$. Figure~\ref{BOEMsupp:fig:blocksizes} shows the estimation of $\phi$ over $100$ independent Monte Carlo runs (same conclusions could be drawn for $\sigma_{u}^{2}$ and $\sigma_{v}^{2}$). For each choice of $\{\tau_{n}\}_{n\ge 0}$, the median and first and last quartiles of the estimation are represented as a function of the number of observations.
    
  We observe that BOEM does not converge when the block size sequence is
  constant and small: as shown in Figure~\ref{BOEMsupp:fig:blocksizes}, if the number of
  observations is too small ($\tau_{n}=25$), the algorithm is a poor
  approximation of the {\em limiting EM} recursion and does not converge. With
  greater block sizes ($\tau_{n}=100$ or $\tau_{n}=350$), the algorithm
  converges but the convergence is slower because it is initialized far from
  the true value and many observations are needed to get several estimations.
  BOEM with slowly increasing block sizes has a better behavior since many
  estimations are produced at the beginning and, once the estimates are closer
  to the true value, the bigger block sizes reduce the variance of the
  estimation.

  Moreover, our convergence rates are given up to a multiplicative constant~:
  the theory says that $\sum_n \tau_n^{-\gamma/2} < \infty$ where $\gamma$ is
  related to the ergodic behavior of the HMM (see
  assumptions H\ref{BOEMsupp:assum:size-block}).
  
  Even if the sequence is chosen to increase at a polynomial rate, we can have
  $\tau_{n}\sim c \ n^{\alpha}$ ($\alpha>1$) with a constant $c$ such that the
  first blocks are quite small to allow a sufficiently large number of updates
  of the parameters $\{\theta_n, n\geq 1 \}$. During a (deterministic)
  "burn-in" period, the first blocks can even be of a fixed length before
  beginning the ``increasing'' procedure.
  
\begin{figure}[!h]
   \centering
   \subfloat[$\tau_{n}= n^{1.1}$ (red) and $\tau_{n}=25$ (blue).]{\includegraphics[width=0.5\textwidth]{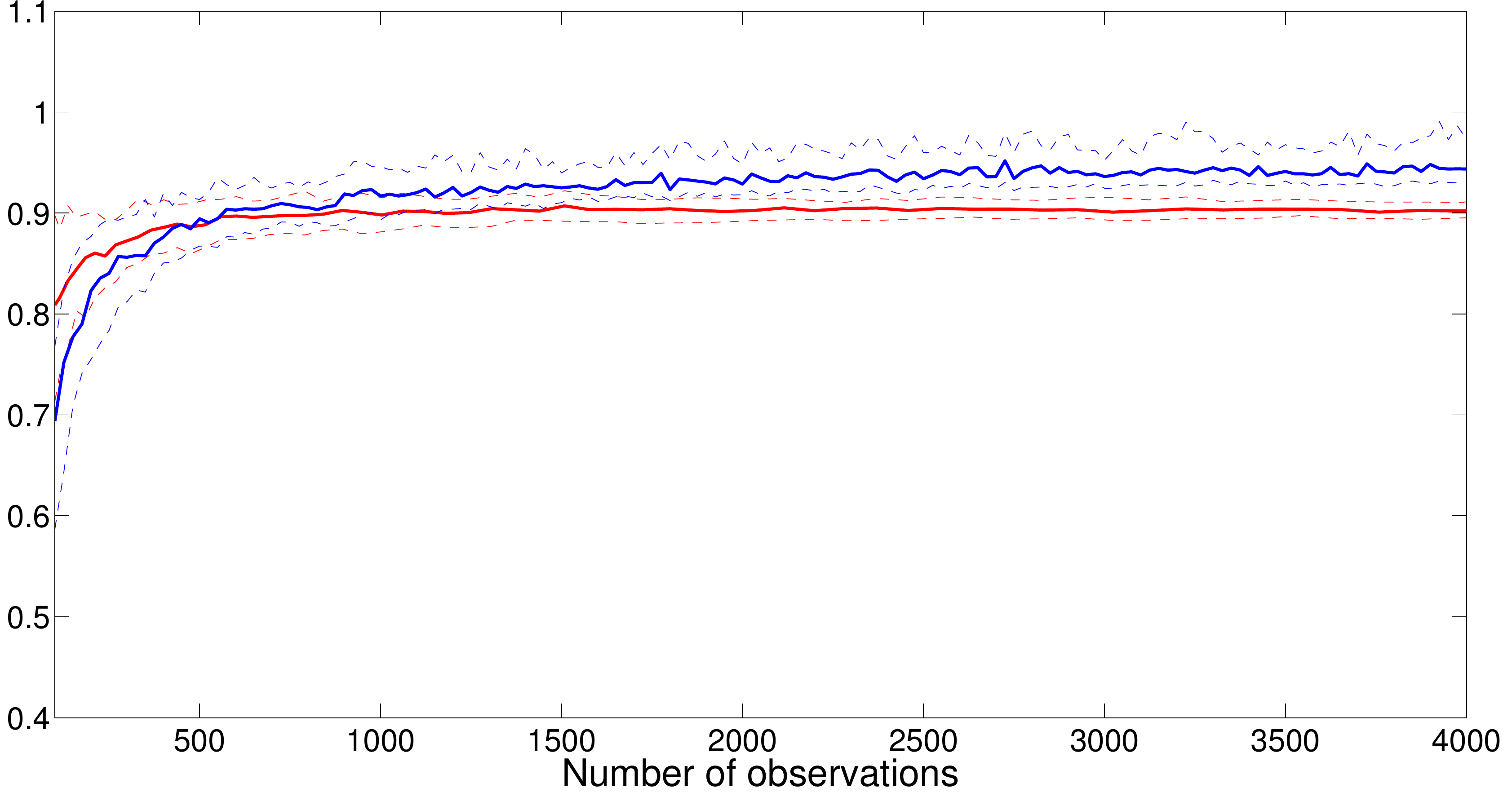}}
   \subfloat[$\tau_{n}= n^{1.1}$ (red) and $\tau_{n}=100$ (blue).]{\includegraphics[width=0.5\textwidth]{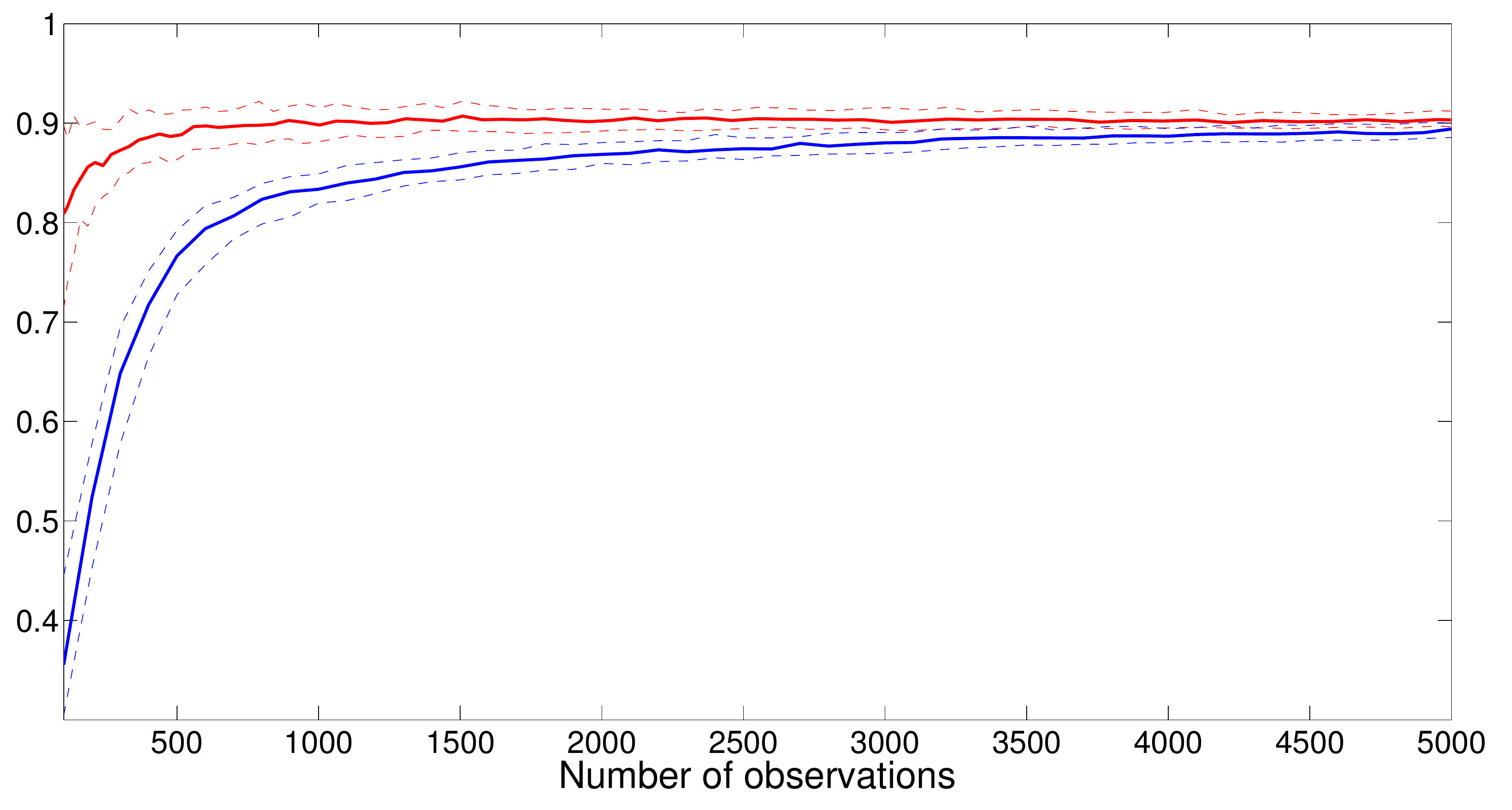}}\\
   \subfloat[$\tau_{n}= n^{1.1}$ (red) and $\tau_{n}=350$ (blue).]{\includegraphics[width=0.6\textwidth]{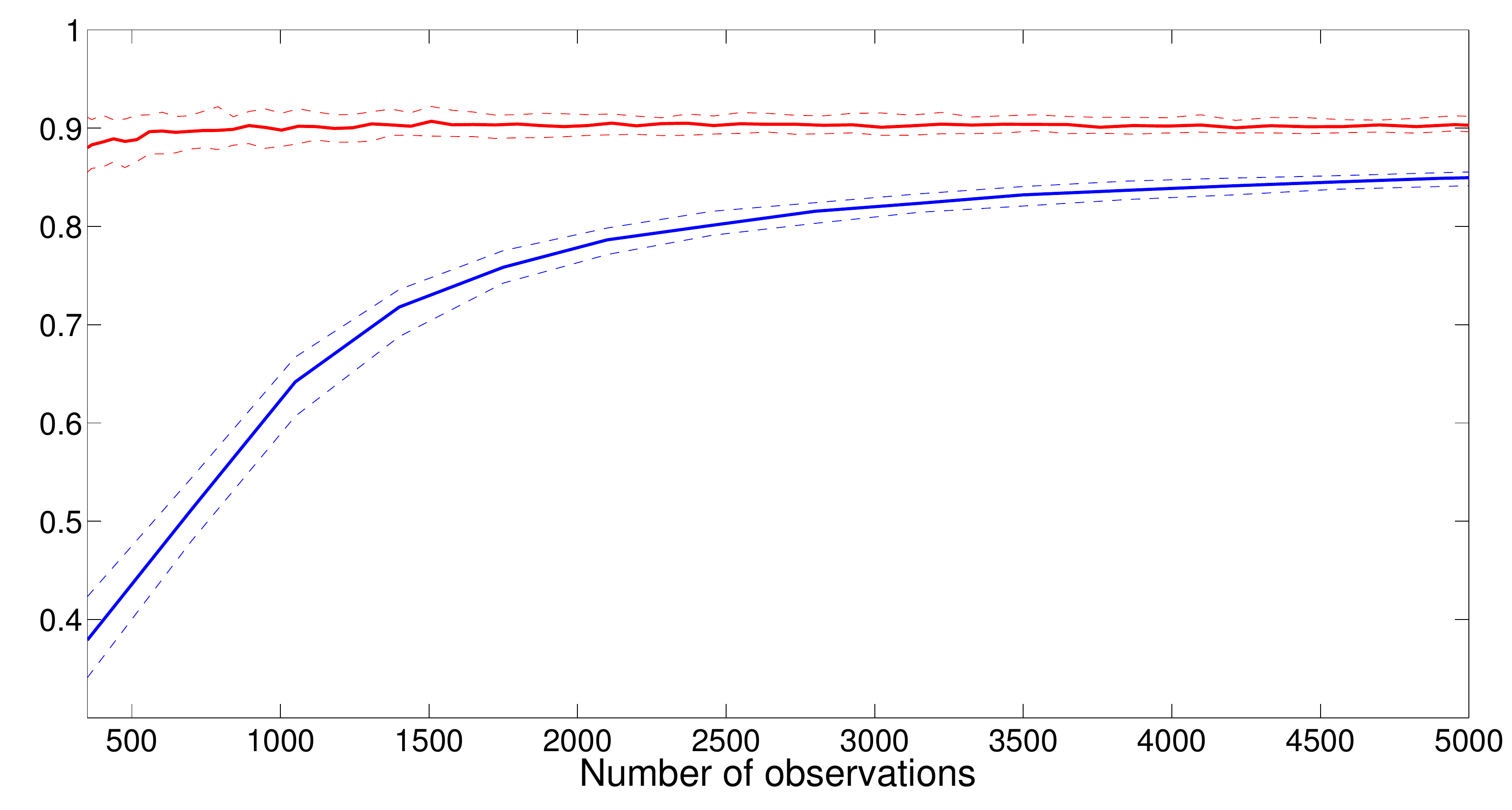}}
   \caption{Estimation of  $\phi$  with different block size  schemes: the median (bold line) and the first and last quartiles (dotted line) are shown for $\tau_{n}= n^{1.1}$ (red), $\tau_{n}=100$ (black) and $\tau_{n} = 350$ (purple). The quantities are computed with $100$ Monte Carlo runs.}
   \label{BOEMsupp:fig:blocksizes}
 \end{figure}
 
\subsection{Finite state-space HMM}
\label{BOEMsupp:sec:finiteHMM}
Observations are
sampled using $d=6$, $v = 0.5$, $x_{i} = i\eqsp, \forall i\in\{1,\dots,d\}$ and the true transition matrix is given by
\[
m =  \begin{pmatrix}0.5 & 0.05 & 0.1 & 0.15 & 0.15 & 0.05\\ 
    0.2& 0.35  &0.1 &0.15 &0.05& 0.15\\
          0.1& 0.1  &0.6 & 0.05 &0.05 &0.1\\
          0.02 &0.03& 0.1 &0.7 &0.1& 0.05\\
          0.1 &0.05& 0.13 &0.02 &0.6 &0.1\\
          0.1& 0.1& 0.13 &0.12& 0.1& 0.45\end{pmatrix}\eqsp.
\]
 \subsubsection{Comparison to an online EM based procedure}
 In this case, we want to estimate the states $\{x_{1}, \dots, x_{d}\}$. All the runs are started from $v = 2$ and from the initial states $\{-1;0;.5;2;3;4\}$.  The experiment is the same as the one in \cite[Section~$3.2$]{lecorff:fort:2011}. The averaged BOEM is compared to an online EM procedure (see \cite{cappe:2011}) combined with Polyak-Ruppert averaging (see \cite{polyak:1990}). This online EM based algorithm follows a stochastic approximation update and depends on a step-size sequence $\{\gamma_{n}\}_{n\geq 0}$ which is chosen in the same way as in \cite[Section~$3.2$]{lecorff:fort:2011}.  Figure~\ref{BOEMsupp:fig:quantilex2} displays the empirical median and first and last quartiles for the estimation of $x_{2}$ with both averaged algorithms as a function of the number of observations. These estimates are obtained over $100$ independent Monte Carlo runs with $\tau_{n} = n^{1.1}$ and $\gamma_{n} = n^{-0.53}$. Both algorithms converge to the true value $x_{2} = 2$ and these plots confirm the similar behavior of BOEM and the online EM of \cite{cappe:2011}.
  \begin{figure}[!h]
   \centering
   \subfloat[Estimation of $x_{2}$ with averaged BOEM.]{\includegraphics[width=0.5\textwidth]{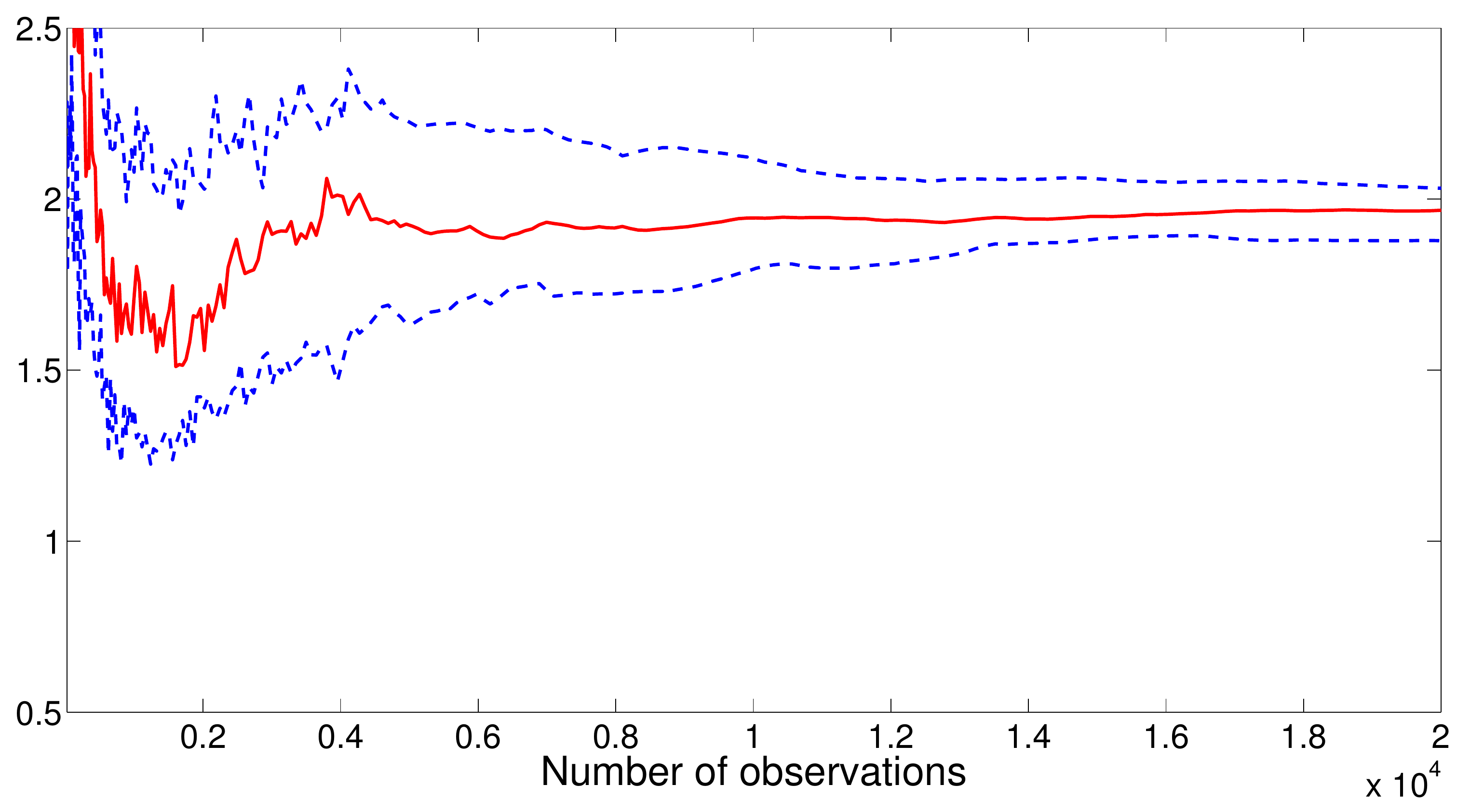}}
   \subfloat[Estimation of $x_{2}$ with averaged OEM.]{\includegraphics[width=0.5\textwidth]{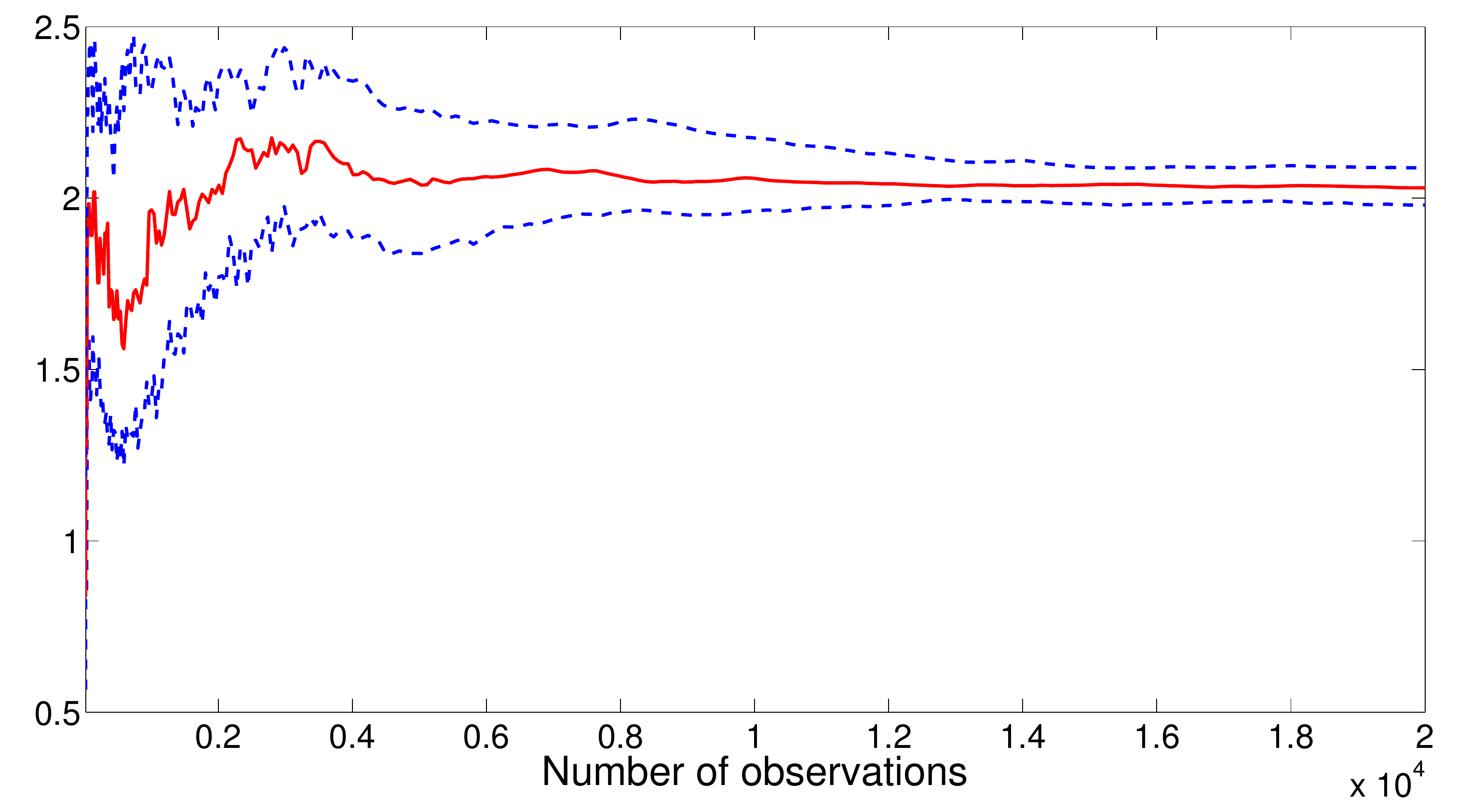}}
   \caption{Estimation of $x_{2}$ using the averaged online EM and averaged BOEM. Each plot displays the empirical median (bold line) and the first and last quartiles (dotted lines) over $100$ independent Monte Carlo runs with $\tau_{n} = n^{1.1}$ and $\gamma_{n} = n^{-0.53}$. The first ten observations are omitted for a better visibility.}
   \label{BOEMsupp:fig:quantilex2}
 \end{figure}
 
\subsubsection{Comparison to a recursive maximum likelihood procedure}
In the numerical applications below, we give supplementary graphs to compare the convergence of the averaged BOEM with the convergence of the Polyak-Ruppert averaged RML procedure. The experiment is the same as the one in \cite[Section~$3.2$]{lecorff:fort:2011}. Figure~\ref{BOEMsupp:fig:v} and \ref{BOEMsupp:fig:q12} displays the empirical median and first and last quartiles of the estimation of  $v$ and $m(1,2)$ over $100$ independent Monte Carlo runs. Both algorithms have a similar behavior for the estimation of these parameters.
  
\begin{figure}[!h]
   \centering
   \subfloat[Averaged BOEM.]{\includegraphics[width=0.5\textwidth]{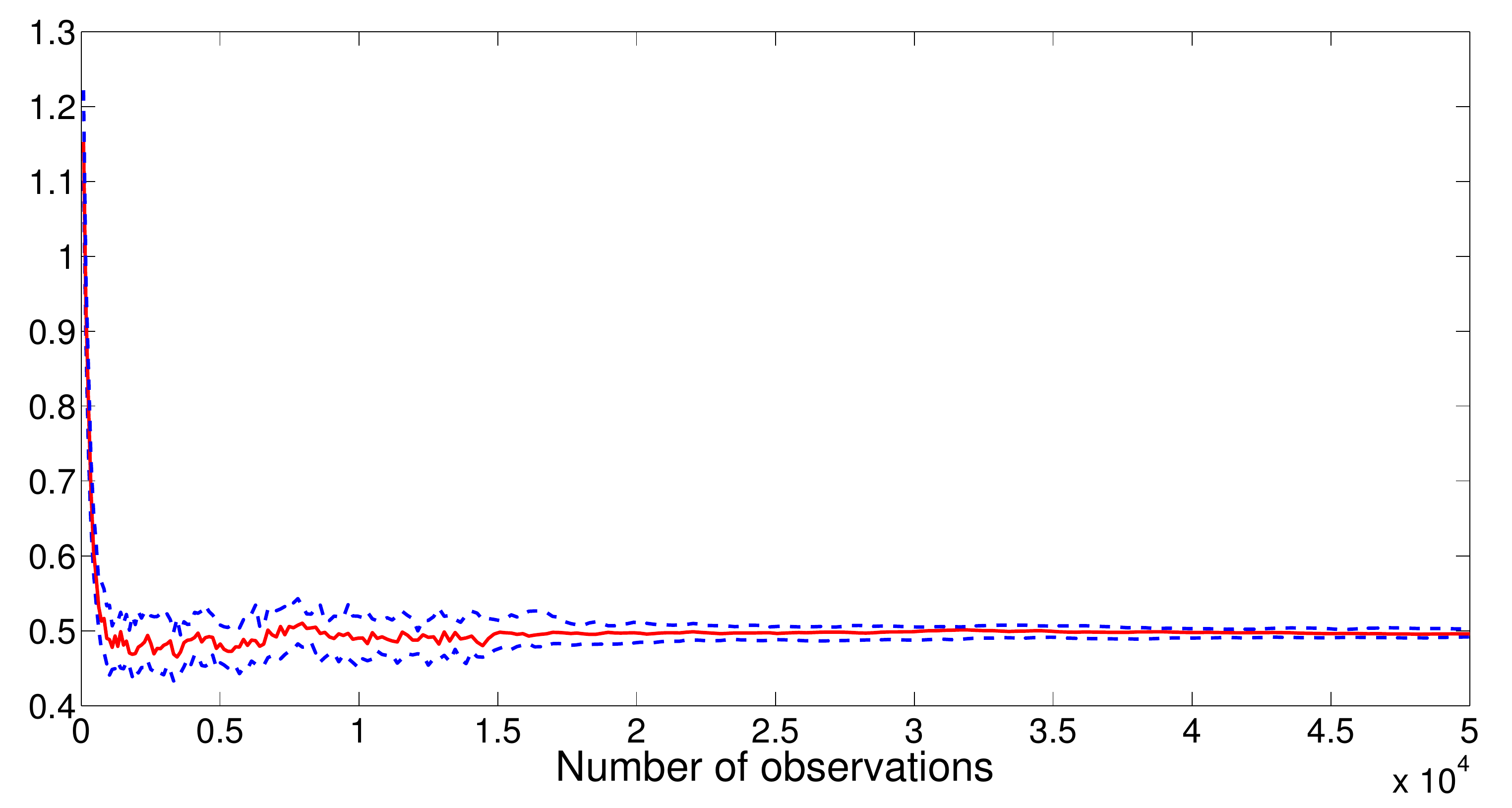}}
   \subfloat[Averaged RML.]{\includegraphics[width=0.5\textwidth]{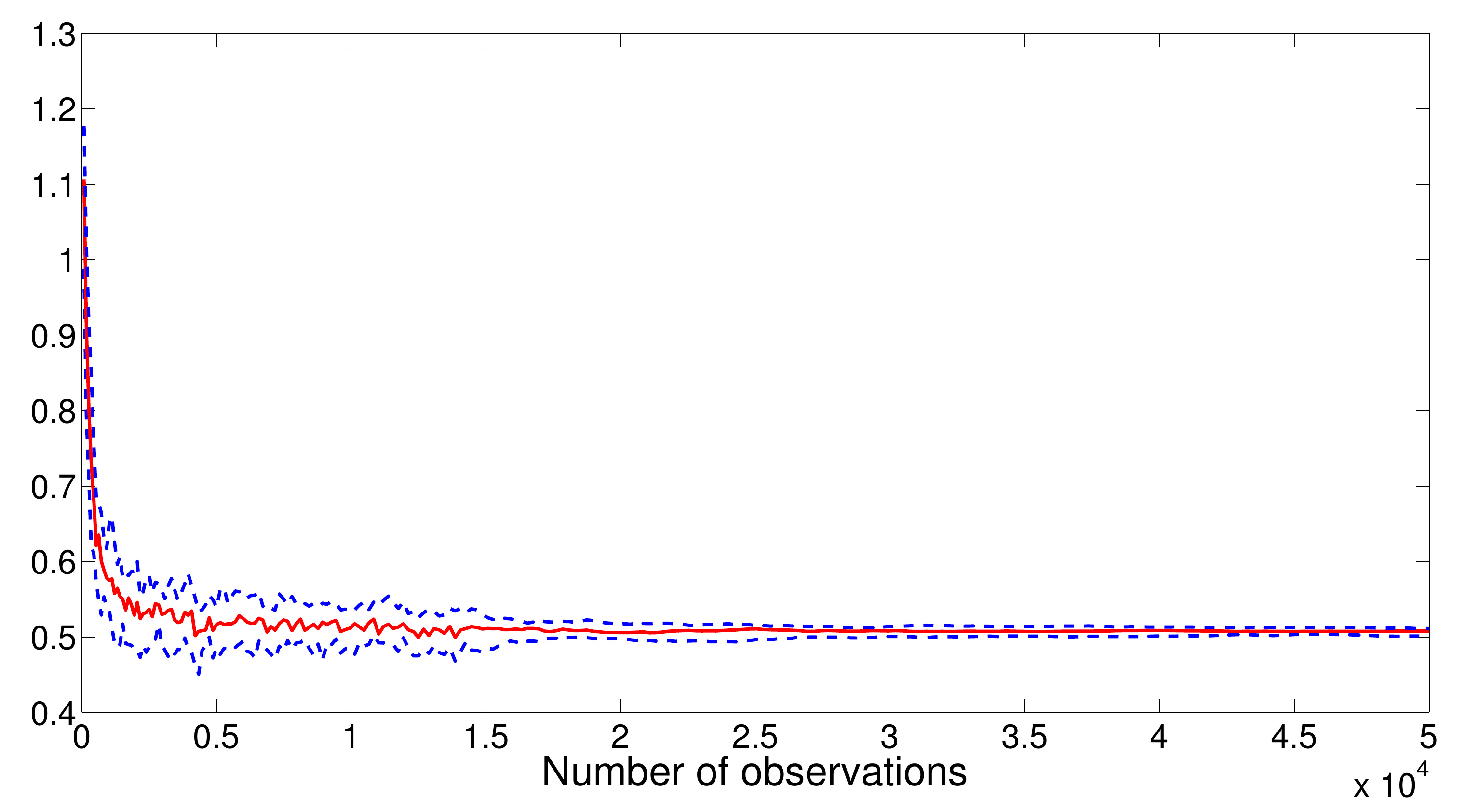}}
   \caption{Empirical median (bold line) and first and last quartiles (dotted line) for the estimation of $v$ using the averaged RML algorithm (right)  and the averaged BOEM algorithm (left). The true values is $v=0.5$ and the averaging procedure is starter after $10000$ observations. The first $10000$ observations are not displayed for a better clarity.}
   \label{BOEMsupp:fig:v}
 \end{figure}

\begin{figure}[!h]
   \centering
   \subfloat[Averaged BOEM.]{\includegraphics[width=0.5\textwidth]{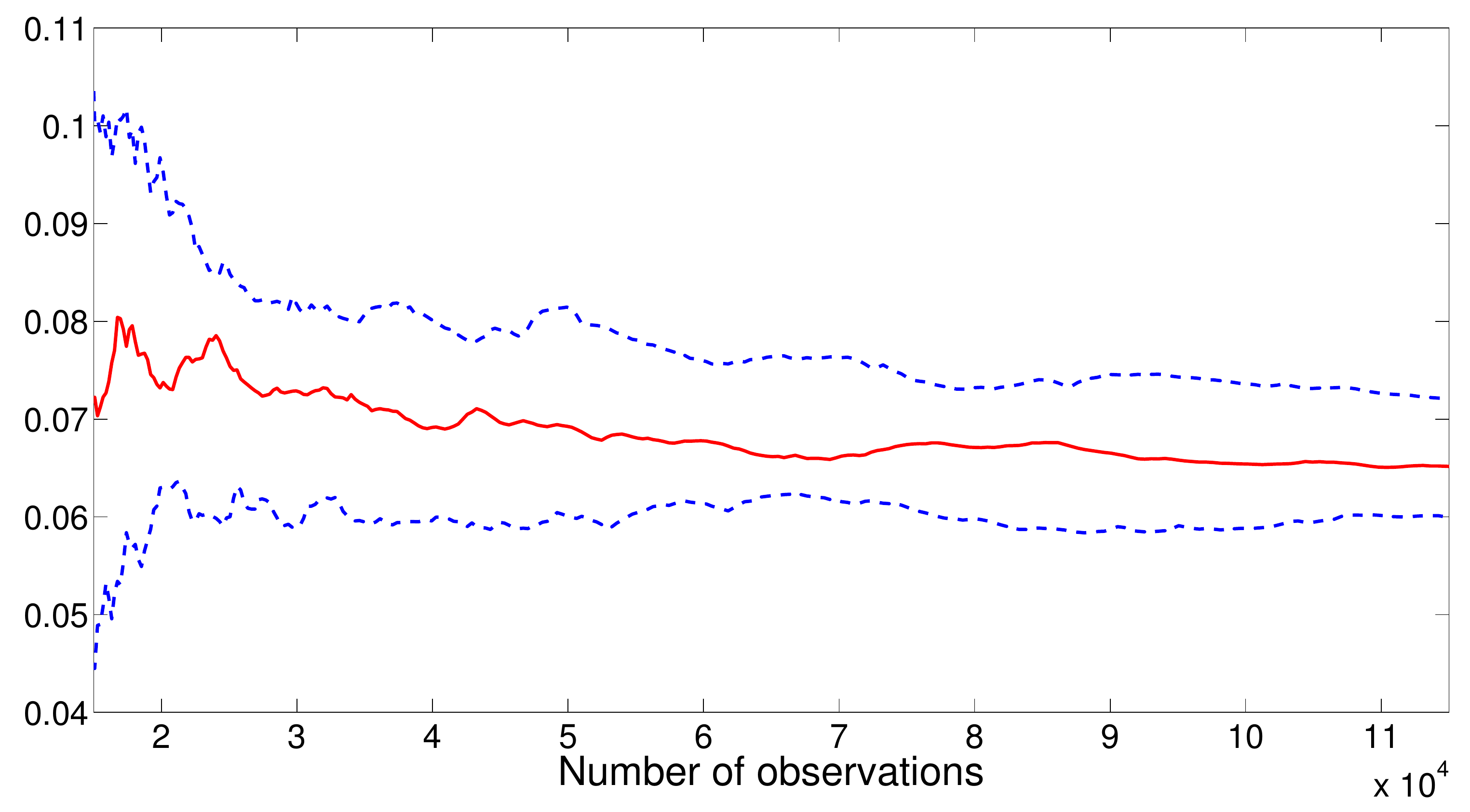}}
   \subfloat[Averaged RML.]{\includegraphics[width=0.5\textwidth]{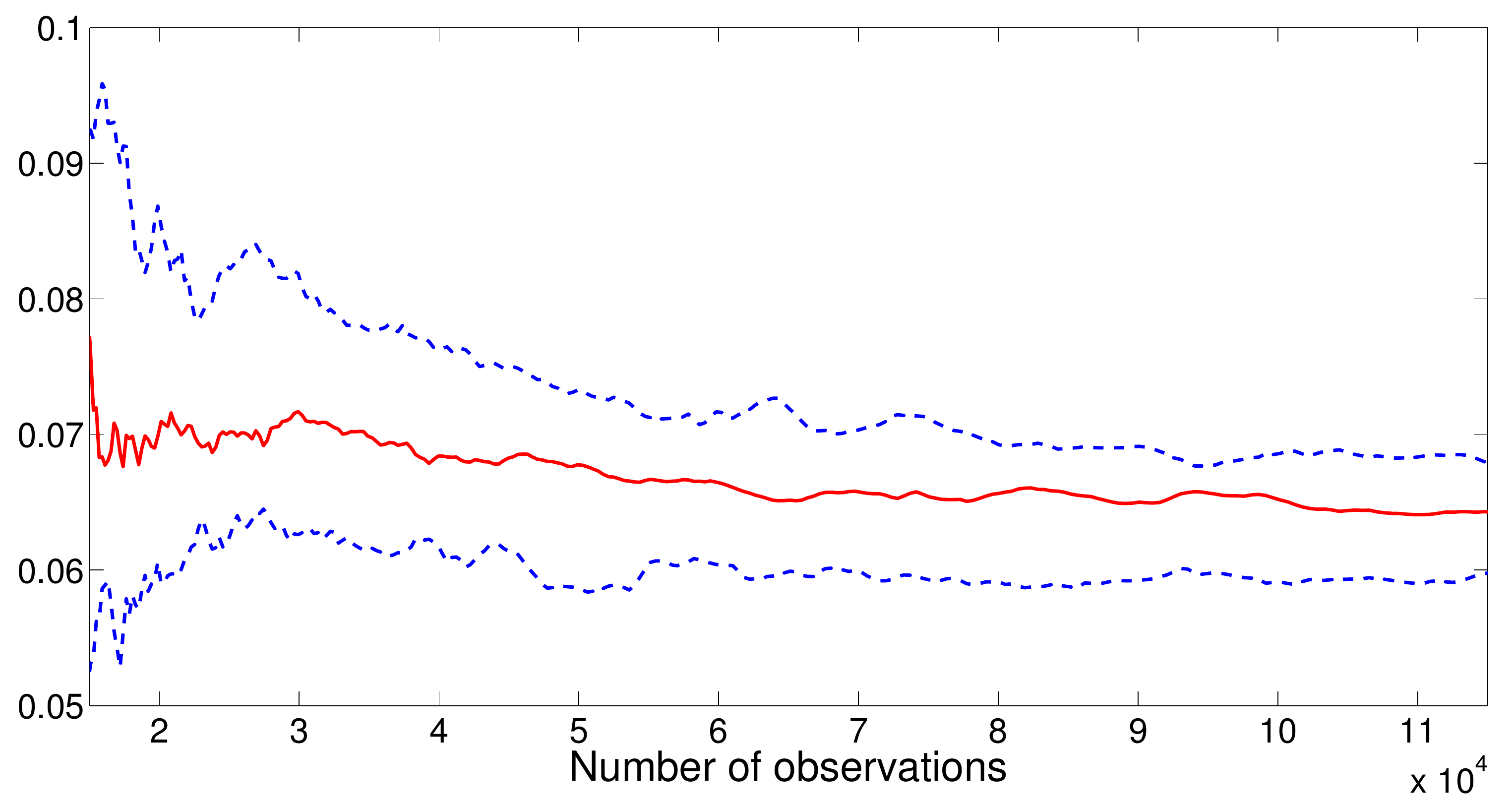}}
   \caption{Empirical median (bold line) and first and last quartiles (dotted line) for the estimation of $m(1,2)$ using the averaged RML algorithm (right)  and the averaged BOEM algorithm (left). The true values is $m(1,2)=0.05$ and the averaging procedure is starter after $10000$ observations. The first $10000$ observations are not displayed for a better clarity.}
   \label{BOEMsupp:fig:q12}
 \end{figure}
  
\subsection{Stochastic volatility model}
\label{BOEMsupp:sec:finiteSVM}
Consider the following stochastic volatility model:
\begin{equation*}
X_{t+1} = \phi X_t + \sigma U_{t}\eqsp, \qquad \qquad
Y_t = \beta \rme^{\frac{X_t}{2}} V_t\eqsp,
\end{equation*}
where $X_0\sim\mathcal{N}\left(0, (1-\phi^2)^{-1} \sigma^2\right)$ and
$(U_t)_{t\geq 0}$ and $(V_t)_{t\geq 0}$ are two sequences of i.i.d.  standard
Gaussian r.v., independent from $X_0$. Data are sampled using $\phi = 0.8$,
$\sigma^{2} = 0.2$ and $\beta^{2} = 1$. All runs are started with $\phi = 0.1$,
$\sigma^{2} = 0.6$ and $\beta^{2} = 2$.

In this model, the smoothed sufficient statistics $\{\bar
S_{\tau_n}^{\chi,T_{n-1}}(\param_{n-1}, \bfY)\}_{n\geq 1}$ can not be computed
explicitly. We thus propose to replace the exact computation by a Monte Carlo
approximation based on particle filtering.  The performance of the Stochastic
BOEM is compared to the online EM algorithm given in~\cite{cappe:2011} (see
also \cite{delmoral:doucet:singh:2010a}).  To our best knowledge, there do not
exist results on the asymptotic behavior of the algorithms
by~\cite{cappe:2011,delmoral:doucet:singh:2010a}; these algorithms rely on many
approximations that make the proof quite difficult (some insights on the
asymptotic behavior are given in~\cite{cappe:2011}). Despite there are no
results in the literature on the rate of convergence of the Online EM algorithm
by \cite{cappe:2011} we choose the step size $\gamma_n$ in \cite{cappe:2011}
and the block size
$\tau_n$ s.t.  $\gamma_n =n^{-0.6}$ and $\tau_n \propto n^{3/2}$ (see
\cite[Section~$3.2$]{lecorff:fort:2011} for a discussion on this choice).  $50$ particles are used for the
approximation of the filtering distribution by Particle filtering.  We report
in Figure~\ref{BOEMsupp:fig:boxSVM}, the boxplots for the estimation of the three
parameters $(\beta, \phi, \sigma^{2})$ for the Polyak-Ruppert
\cite{polyak:1990} averaged Online EM and the averaged BOEM.  Both average
versions are started after $20000$ observations. 
Figure~\ref{BOEMsupp:fig:boxSVM} displays the estimation of $\phi$, $\sigma^{2}$ and $\beta^{2}$.  This figure shows that both algorithms
have the same behavior. Similar conclusions are obtained by considering other
true values for $\phi$ (such as $\phi = 0.95$). Therefore, the intuition
  is that online EM and Stochastic BOEM have the same asymptotic  behavior.
  The main advantage of the second approach is that it relies on approximations
  which can be controlled in such a way that we are able to show that the
  limiting points of the particle version of the Stochastic BOEM algorithms are
  the stationary points of the limiting normalized log-likelihood of the
  observations.
    
\begin{figure}[!h]
  \centering
  \subfloat[Estimation of $\phi$.]{\label{BOEMsupp:fig:boxSVMphi}\includegraphics[width=0.7\textwidth]{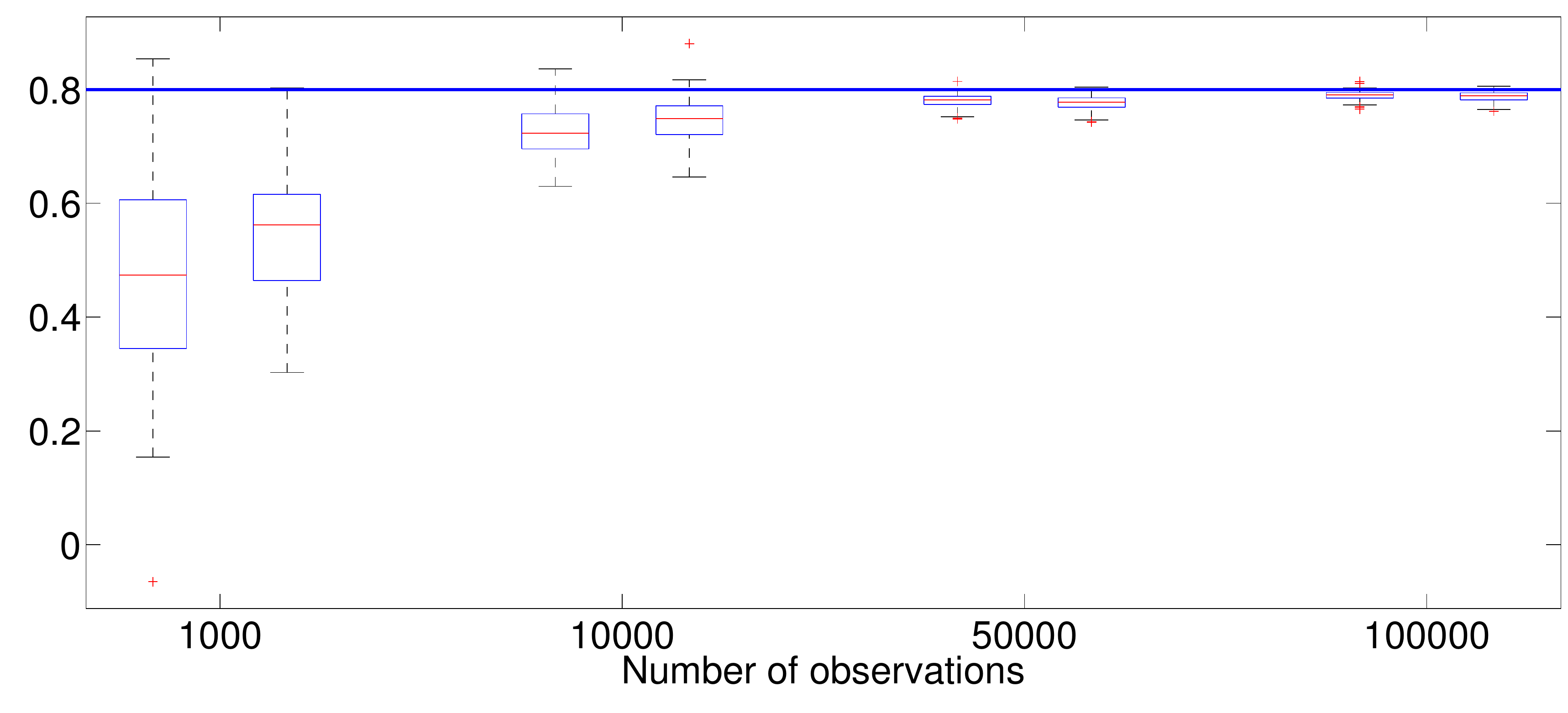}}\\
  \subfloat[Estimation of $\sigma^{2}$.]{\label{BOEMsupp:fig:boxSVMsigma2}\includegraphics[width=0.7\textwidth]{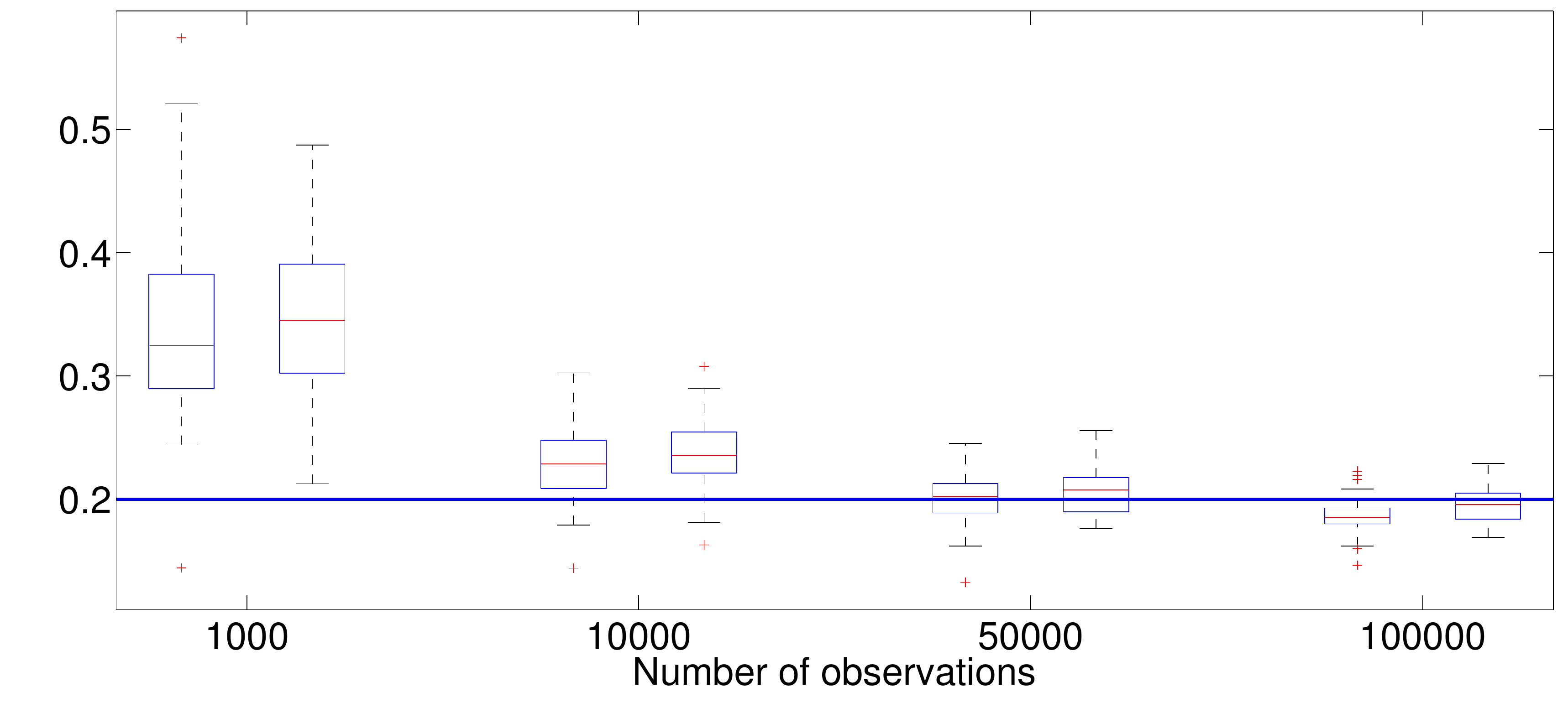}}\\
  \subfloat[Estimation of $\beta^{2}$.]{\label{BOEMsupp:fig:boxSVMbeta2}\includegraphics[width=0.7\textwidth]{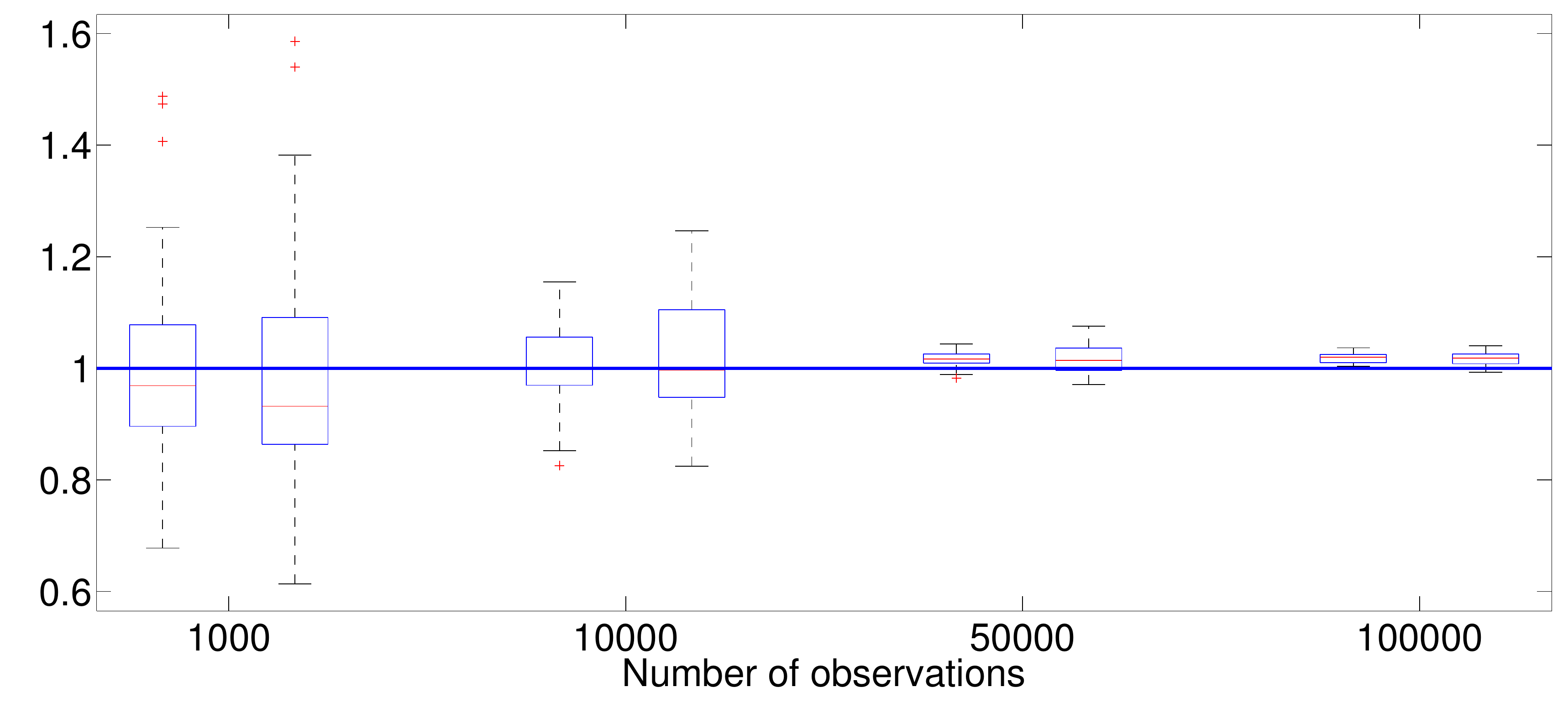}}
   \caption{Estimation of $\phi$, $\sigma^{2}$ and $\beta^{2}$ using the averaged online EM algorithm (left) and the averaged BOEM (right), after  $n=\{1000, 10k, 50k, 100k \}$ observations. The true value of $\phi$ is $0.8$.}
   \label{BOEMsupp:fig:boxSVM}
  \end{figure}
  
  We now compare the two algorithms when the true value of $\phi$ is (in
  absolute value) closer to $1$: we choose $\phi = 0.95$, $\beta^2$ and
  $\sigma^2$ being the same as in the previous experiment.
  
  As illustrated on Figure~\ref{BOEMsupp:fig:boxSVM2}, the same conclusions are
  drawn for greater values of $\phi$.
\begin{figure}[]
  \centering
 \label{fig:boxSVMphi}\includegraphics[width=0.8\textwidth]{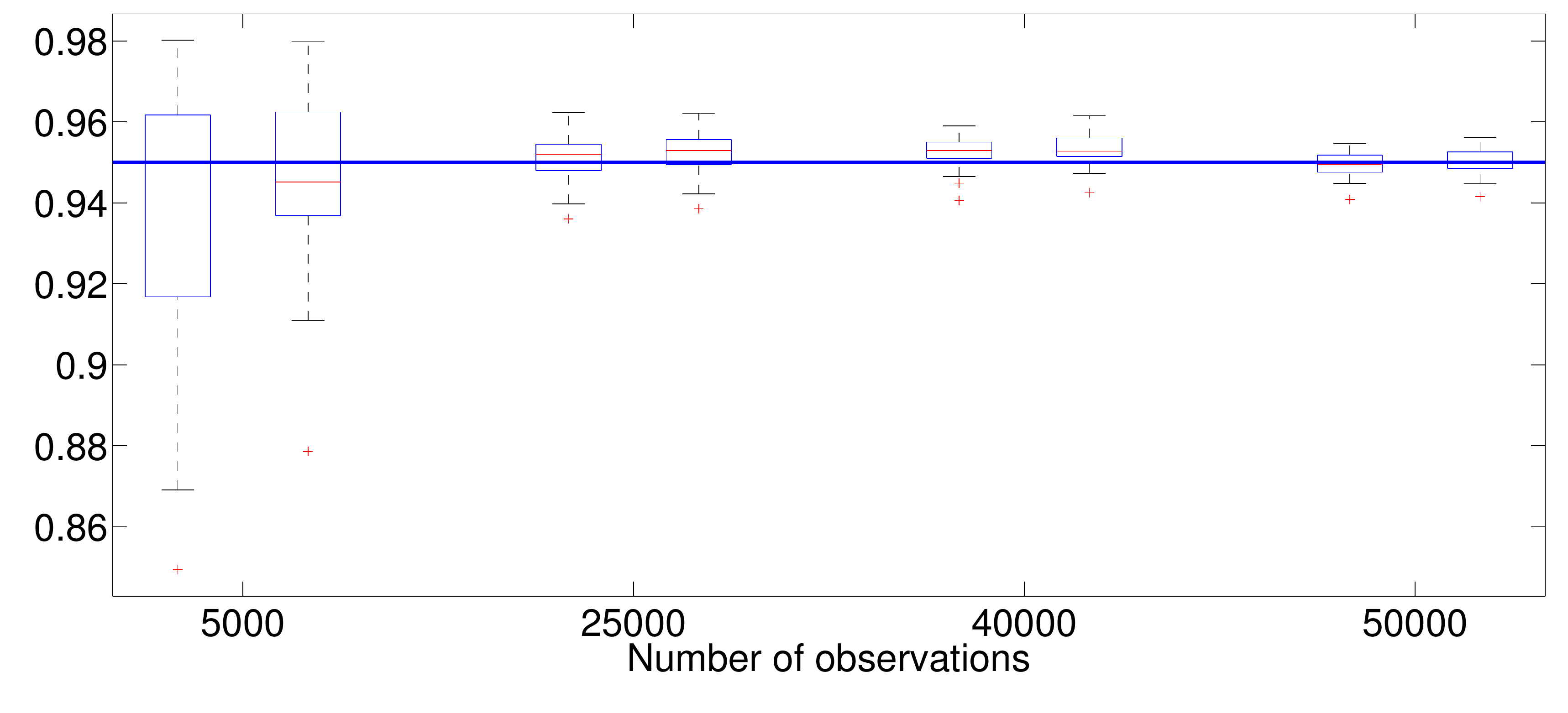}
   \caption{Estimation of $\phi$ using the averaged online EM algorithm (left) and the averaged BOEM algorithm (right), after $n=\{5k, 25k, 40k, 50k \}$ observations. The true value of $\phi$ is $0.95$.}
   \label{BOEMsupp:fig:boxSVM2}
  \end{figure}

\bibliographystyle{plain} 
\bibliography{./onlineblock}

\begin{thebibliography}{1}

\bibitem{billingsley:1987}
P.~Billingsley.
\newblock {\em Probability and Measure}.
\newblock Wiley, New York, 3rd edition, 1995.

\bibitem{cappe:2011}
O.~Capp\'e.
\newblock Online {EM} algorithm for {H}idden {M}arkov {M}odels.
\newblock {\em {J}. {C}omput. {G}raph. {S}tatist.}, 20(3):728--749, 2011.

\bibitem{cappe:moulines:ryden:2005}
O.~Capp\'{e}, E.~Moulines, and T.~Ryd\'{e}n.
\newblock {\em Inference in Hidden {M}arkov Models}.
\newblock Springer, 2005.

\bibitem{davidson:1994}
J.~Davidson.
\newblock {\em Stochastic Limit Theory: An Introduction for Econometricians}.
\newblock Oxford University Press, 1994.

\bibitem{delmoral:doucet:singh:2010a}
M.~Del~Moral, A.~Doucet, and S.S Singh.
\newblock Forward smoothing using sequential {M}onte {C}arlo.
\newblock arXiv:1012.5390v1, Dec 2010.

\bibitem{douc:moulines:ryden:2004}
R.~Douc, E.~Moulines, and T.~Ryd\'{e}n.
\newblock Asymptotic properties of the maximum likelihood estimator in
  autoregressive models with {M}arkov regime.
\newblock {\em Ann. Statist.}, 32(5):2254--2304, 2004.

\bibitem{lecorff:fort:2011}
S.~Le~Corff and G.~Fort.
\newblock {O}nline {E}xpectation {M}aximization based algorithms for inference
  in {H}idden {M}arkov {M}odels.
\newblock Technical report, arXiv:1108.3968, 2011.

\bibitem{polya:1976}
G.~P{\'o}lya and G.~Szeg{\H{o}}.
\newblock {\em Problems and Theorems in Analysis. {V}ol. {II}}.
\newblock Springer, 1976.

\bibitem{polyak:1990}
B.~T. Polyak.
\newblock A new method of stochastic approximation type.
\newblock {\em Autom. Remote Control}, 51:98--107, 1990.

\end{thebibliography}
\end{document}